\pdfoutput=1
\documentclass[11pt]{amsart}
\usepackage{amsmath}
\usepackage{amssymb}
\usepackage{hyperref} 
\hypersetup{colorlinks=true, linkcolor=blue, citecolor = blue, urlcolor=blue, filecolor=magenta}
\usepackage{enumerate}
\usepackage{slashed}
\usepackage{caption}
\usepackage{subcaption}
\usepackage{newlfont}
\usepackage{amsrefs}
\usepackage{comment}
\usepackage[normalem]{ulem}
\usepackage{accents}
\usepackage{leftidx}

\usepackage{harpoon}
\usepackage[pdftex]{graphicx}
\usepackage{esint}
\usepackage{relsize}
\usepackage[latin1]{inputenc}
\usepackage{xcolor}

\definecolor{brass}{rgb}{0.71, 0.65, 0.36}

\usepackage{tikz-cd}
\theoremstyle{plain}
\newtheorem{theorem}{Theorem}[section]

\newtheorem{prop}[theorem]{Proposition}

\theoremstyle{definition}
\newtheorem{remark}[theorem]{Remark}

\numberwithin{equation}{section}

\newcommand{\td}{\text{d}}
\theoremstyle{plain}

\numberwithin{equation}{section}

\allowdisplaybreaks

\begin{document}
\title[The Spacetime Penrose Inequality for Cohomogeneity One Initial Data]{The Spacetime Penrose Inequality for Cohomogeneity One Initial Data}

\author[Khuri]{Marcus Khuri}
\address{Department of Mathematics\\
Stony Brook University\\
Stony Brook, NY 11794, USA}
\email{marcus.khuri@stonybrook.edu}

\author[Kunduri]{Hari Kunduri}
\address{Department of Mathematics \& Statistics\\ and Department of Physics \& Astronomy\\
McMaster University\\ Hamilton, ON L8S 4M1, Canada}
\email{kundurih@mcmaster.ca}


\thanks{M. Khuri acknowledges the support of NSF Grant DMS-2104229. H. Kunduri acknowledge the support of NSERC Grant RGPIN-2018-04887.}

\begin{abstract}
We prove the spacetime Penrose inequality for asymptotically flat $2(n+1)$-dimensional initial data sets for the Einstein equations, which are invariant under a cohomogeneity one action of $\mathrm{SU}(n+1)$. Analogous results are obtained for asymptotically hyperbolic initial data that arise as spatial hypersurfaces in asymptotically Anti de-Sitter spacetimes. More precisely, it is shown that with the dominant energy condition, the total mass is bounded below by an explicit function of the outermost apparent horizon area. Furthermore, the inequality is saturated if and only if the initial data isometrically embed into a Schwarzschild(-AdS) spacetime. 
This generalizes the only previously known case of the conjectured spacetime Penrose inequality, established under the assumption of spherical symmetry.
Additionally, in the time symmetric case, we observe that the inequality holds for $4(n+1)$-dimensional and 16-dimensional initial data invariant under cohomogeneity one actions of $\mathrm{Sp}(n+1)$ and $\mathrm{Spin}(9)$, respectively, thus treating the inequality for all cohomogeneity one actions in this regime.
\end{abstract}

\maketitle

\section{Introduction}
\label{sec1} \setcounter{equation}{0}
\setcounter{section}{1}

In an effort to find a counterexample to the weak cosmic censorship conjecture \cite{Penrose}, Penrose put forward a precise inequality \cite{Penrose1} relating the ADM mass $m$ of an asymptotically flat 4-dimensional spacetime to any cross-sectional area $\mathcal{A}_{e}$ of the event horizons it contains, in the form
\begin{equation}
m\geq\sqrt{\frac{\mathcal{A}_e}{16\pi}}.
\end{equation}
It is typical to reformulate this inequality in the setting of initial data sets. Consider a triple $(M^d,g,k)$ consisting of a $d$-dimensional connected and orientable manifold $M^d$ with boundary, a complete Riemannian metric $g$, and a symmetric 2-tensor $k$ denoting the extrinsic curvature of an embedding into spacetime, with all objects being smooth. These quantities must satisfy the constraint equations
\begin{equation}\label{densityem}
16 \pi \mu=R+(\text{Tr}_g k)^2-|k|_g^2, \qquad\text{ }
8\pi J=\text{div}_g\left(k-(\text{Tr}_g k)g\right),
\end{equation}
where $R$ is scalar curvature, and $\mu$, $J$ represent the energy-momentum density of matter fields. We will say that the \textit{dominant energy condition} is satisfied if $\mu\ge |J|_g$. Moreover, the data will be referred to as \textit{asymptotically flat} if outside a compact set $\mathcal{C}$ there is a diffeomorphism $\varphi: M^d \setminus \mathcal{C}\rightarrow\mathbb{R}^d\setminus B_1$, such that in the Caretesian coordinates $x$ provided by this map
\begin{equation}\label{1.3}
\varphi_{*} g-\delta=O_2(|x|^{-\tau}), \quad
 \varphi_{*} k= O_1(|x|^{-\tau-1}),\quad
\mu,J=O(|x|^{-2\tau-2}),\quad \mathrm{Tr}_g k=O_1(|x|^{-2\tau-1}),
\end{equation}
for some $\tau>\tfrac{d-2}{2}$. The additional decay on the trace of $k$ is usually not included in the definition of asymptotically flatness, but will be useful when working with the generalized Jang equation below.
With these asymptotics the ADM energy and linear momentum are well-defined \cites{Bartnik,Chrusciel} and given by
\begin{equation}
E\!=\!\frac{1}{2(d-1)\omega_{d-1}}\lim_{r\rightarrow\infty}
\int_{S_{r}}(g_{ij,i}-g_{ii,j})\nu^{j}dV,\quad
P_i \!=\!\frac{1}{(d-1)\omega_{d-1}}\lim_{r\rightarrow\infty}
\int_{S_{r}}(k_{ij}-(\text{Tr}_g k)g_{ij})\nu^{j}dV,
\end{equation}
where $S_r$ are coordinate spheres with unit outer normal $\nu$ and $\omega_{d-1}$ is the volume of the unit $(d-1)$-sphere. 
The ADM mass is then the Lorentz length of the energy-momentum vector, $m=\sqrt{E^2-|P|^2}$.

The role of the event horizon is replaced by that of an apparent horizon, which may be computed directly from the initial data. Recall that the gravitational field's strength
near a hypersurface $\Sigma\subset M^d$ may be probed by the null expansions (null mean curvatures) 
\begin{equation}
\theta_{\pm}=H_{\Sigma}\pm \text{Tr}_{\Sigma}k,
\end{equation}
where $H_{\Sigma}$ denotes the mean curvature with respect to the 
normal pointing towards infinity. These give the rate of change for area of a shell of light emanating from the surface in the outward future/past direction ($+$/$-$). Future or past trapped surfaces are defined by the inequalities $\theta_{+}< 0$ or $\theta_{-}< 0$, respectively, and may be interpreted as lying within a region of strong gravity. When $\theta_{+}=0$ or $\theta_{-}=0$ the surface is called a future or past apparent horizon; these naturally arise
as boundaries of future or past trapped regions \cite{AnderssonMetzger}. Furthermore, such a surface will be called an \textit{outermost apparent horizon} if it is not enclosed by any other apparent horizon. The conjectured Penrose inequality for general dimensions may then be recast as
\begin{equation}\label{ofinaoinoinoiqhhq}
m\geq \frac{1}{2}\left(\frac{\mathcal{A}_h}{\omega_{d-1}}\right)^{\frac{d-2}{d-1}}
\end{equation}
whenever the dominant energy condition holds, where $\mathcal{A}_h$ is the smallest area required to enclose the outermost apparent horizon. Equality should be achieved only for slices of the Schwarzschild spacetime.

In the (Riemannian) time symmetric case when $k=0$, the 3-dimensional Penrose inequality has been confirmed by Huisken-Ilmanen \cite{Huisken:2001} and Agostiniani-Mantegazza-Mazzieri-Oronzio \cite{Mazzierietal} for a single black hole via inverse mean curvature flow and $p$-harmonic functions repsectively, and by Bray \cite{Bray} for multiple black holes using a conformal flow. The latter approach has been generalized by Bray-Lee \cite{Bray:2007opu} up to dimension 7. Within the context of the general spacetime setting, there are very few results. In fact, for this regime the conjectured inequality has only been verified
in the spherically symmetric case \cites{Hayward,MM} with the rigidity statement also obtained in \cites{Bray:2009au,BKS}, \cite[Theorem 7.46]{Leetext}; these results hold in all dimensions. 

In the present note we consider cohomogeneity one initial data sets. Recall that a Riemannian manifold $(M^d,g)$ is said to have \textit{cohomogeneity one} if a (compact connected) Lie group $\mathcal{G}$ acts by isometries on $M^d$ such that the principal orbits $\mathcal{G}/\mathcal{H}$ are of codimension one, where $\mathcal{H}$ is a principal isotropy subgroup (for general initial data sets we also assume that $k$ is invariant under the action of $\mathcal{G}$). 
Within the setting of interest, the manifold will be asymptotically flat or asymptotically hyperbolic with an outermost apparent horizon boundary; note that smooth outermost apparent horizons inherit the symmetries of the initial data from which they arise \cite[Lemma 3.1]{BKS} (\cite[Theorem 8.1]{AMS}, \cite{PS}). Therefore, the principal orbit theorem implies that the orbit space $M^{d} / \mathcal{G}$ is diffeomorphic to a half line $[0,\infty)$ with the origin corresponding to the apparent horizon, and $M^d \cong [0,\infty) \times \mathcal{G} / \mathcal{H}$. This is discussed further in Section \ref{sec2}, where it is also shown that the structure at infinity ensures that the surfaces of homogeneity will be spheres. A classification of the possible homogeneous metrics on spheres has been obtained by Ziller \cite{Ziller}. In addition to the standard round metric, there are 
odd-dimensional cases corresponding to $S^{2n+1} = \mathrm{SU}(n+1)/\mathrm{SU}(n)$, $S^{4n+3} = \mathrm{Sp}(n+1)/ \mathrm{Sp}(n)$, and a homogeneous metric on $S^{15} = \mathrm{Spin}(9) / \mathrm{Spin}(7)$.  We will establish the spacetime Penrose inequality for the first of these cases in the asymptotically flat and asymptotically hyperbolic contexts, namely for initial data of dimension $d=2(n+1)$, $n\geq 1$ which are invariant under the action of $\mathrm{SU}(n+1)$. Note that in contrast to the spherically symmetric case, this class of initial data includes those with 
non-vanishing angular momentum, and there are explicit rotating black hole solutions arising from data in this class, see Appendix \ref{appA}. To accomplish this goal we will exploit a method proposed by Bray and the first author \cites{Bray:2009oni, Bray:2009au} which involves coupling inverse mean curvature flow to the so called generalized Jang equation. In these higher dimensions, lack of the Gauss-Bonnet theorem presents difficulties for monotonicity of Hawking mass, however in the current setting a fortuitous cancellation occurs which allows the procedure to go through. 

\begin{theorem}\label{foiqnoinoiqnhh}
Let $(M^{2(n+1)},g,k)$, $n\geq 1$ be an asymptotically flat $\mathrm{SU}(n+1)$-invariant initial data set, with outermost apparent horizon boundary of area $\mathcal{A}$. If the dominant energy condition is satisfied then 
\begin{equation}\label{foiaoinoqinhoipqjh}
m \geq \frac{1}{2}\left(\frac{\mathcal{A}}{\omega_{2n+1}}\right)^{\frac{2n}{2n+1}},
\end{equation}
and equality occurs if and only if the initial data arise from an isometric embedding into a Schwarzschild spacetime. 
\end{theorem}

Consider now asymptotically hyperbolic initial data relevant for asymptotically Anti-de Sitter (AdS) spacetimes. Let $(\mathbb{H}^d,b)$ denote the $d$-dimensional hyperbolic space with metric expressed in geodesic polar coordinates as $b=dr^2 +(\sinh^2 r) g_{S^{d-1}}$.
Recall that hyperbolic space arises as a totally geodesic spacelike slice of the Anti-de Sitter (AdS) spacetime, $(\mathbb{R} \times \mathbb{H}^d, -(\cosh^2 r) dt^2 + b)$. An initial data set $(M^d,g,k)$ satisfying the modified constraints
\begin{equation}\label{AdSID}
16 \pi \mu=R+(\text{Tr}_g k)^2-|k|_g^2 + d(d-1), \qquad
8\pi J=\text{div}_g\left(k-(\text{Tr}_g k)g\right),
\end{equation} 
will be referred to as {\it asymptotically hyperbolic} if outside a compact set $\mathcal{C}$ there is a diffeomorphism $\varphi : M^d \setminus \mathcal{C} \rightarrow \mathbb{H}^d \setminus B_1$ such that
\begin{equation}\label{AdSdecay}
\mathfrak{g}:=\varphi_* g-b=O_2(e^{-qr}),\qquad \varphi_* k=O_1(e^{-qr}), \qquad 
\mu, J=O(e^{-2q r}), 
\end{equation} 
for $q>d/2$. The dominant energy condition in this setting is again expressed as $\mu\geq|J|_g$. Initial data satisfying \eqref{AdSID} and \eqref{AdSdecay} arise as spacelike hypesurfaces in asymptotically AdS spacetimes with (negative) cosmological constant $\Lambda = -\frac{d(d-1)}{2}$. Well-defined
global Hamiltonian charges, interpreted as the total energy-momentum for such data, were obtained by Chru\'{s}ciel-Nagy \cite{ChruscielNagy} (see also \cite{Michel} and the exposition in \cite{Cederbaum}). In particular, if $U =\cosh r$ denotes the time translation lapse function then the total energy is given by
\begin{equation}\label{CHenergy}
E_{\mathrm{hyp}} =\frac{1}{2(d-1) \omega_{d-1}} \lim_{r \to \infty} \int_{S_r} \left[U(\text{div}_b \mathfrak{g}) - U (d \mathrm{Tr}_b \mathfrak{g}) + (\mathrm{Tr}_b \mathfrak{g}) d U - \mathfrak{g}(\nabla_b U)\right] (\nu_b) d V,
\end{equation} 
where $\nu_b$ is the unit outward normal (measured in $b$) to the coordinate spheres $S_r$. An analogous expression gives rise to the total linear momentum by replacing $U$ with the lapse functions $U_i = x^i \sinh r$, $i=1,\ldots d$ defined on $\mathbb{H}^{d}$, where $x^i$ are Cartesian coordinates restricted to the unit sphere $S^{d-1}$. A version of the spacetime Penrose inequality is conjectured to hold in this asymptotically hyperbolic setting \cites{Bray:2003ns,Cha:2017gej,Gibbons,Mars}, and is relevant in the context of the gauge theory/gravity correspondence \cite{EH}. The inequality has been established in spherical symmetry by Engelhardt-Folkestad \cite[Theorem 6]{EF}, Folkestad \cite[Theorem 1]{F}, and Husain-Singh \cite{HS}, with various hypotheses. Moreover, it has been proved when $k=0$ for small perturbations of the Schwarzschild-AdS manifold by Ambrozio \cite{Amb} (see also \cite{KhuriKopinski}), and for graphs by Dahl-Gicquaud-Sakovich \cite{DGS} and Gir\~{a}o-de Lima \cite{GL1}. We prove the following result.

\begin{theorem}\label{foiqnoinoiqnhh1}
Let $(M^{2(n+1)},g,k)$, $n\geq 1$ be an asymptotically hyperbolic $\mathrm{SU}(n+1)$-invariant initial data set, with outermost apparent horizon boundary of area $\mathcal{A}$. If the dominant energy condition is satisfied then 
\begin{equation}\label{hypPenrose}
E_{\mathrm{hyp}} \geq \frac{1}{2}\left(\frac{\mathcal{A}}{\omega_{2n+1}}\right)^{\frac{2n}{2n+1}} + \frac{1}{2}  \left(\frac{\mathcal{A}}{\omega_{2n+1}}\right)^{\frac{2n+2}{2n+1}},
\end{equation}
and equality occurs if and only if the initial data arise from an isometric embedding into a Schwarzschild-AdS spacetime. 
\end{theorem} 

A different type of asymptotically hyperbolic initial data, with second fundamental form $k$ converging to $g$ rather than vanishing at infinity, appears naturally as asymptotically umbilic slices in asymptotically flat spacetimes and are referred to as asymptotically hyperboloidal.  In this setting, the Penrose inequality \cite{Wang} takes the form \eqref{foiaoinoqinhoipqjh} instead of \eqref{hypPenrose}, and the case of equality should give rise to an embedding into the Schwarzschild spacetime. This has also been established in spherical symmetry by Hou \cite{Hou}, in the time symmetric graphical context by Gir\~{a}o-de Lima \cite{GL}, and we expect that the strategy outlined in \cite{Cha:2015nja} together with arguments of the current paper will yield the $\mathrm{SU}(n+1)$-invariant case as well.

Lastly, in the Riemannian asymptotically flat context, we are able to treat the Penrose inequality for the full range of cohomogeneity one manifolds. In addition to the $\mathrm{SU}(n+1)$-invariant case, this also includes the $\mathrm{Sp}(n+1)$ and $\mathrm{Spin}(9)$-invariant cases. The proof will rely on a combination of Bray's conformal flow with inverse mean curvature flow.

\begin{theorem}\label{thm3}
Let $(M^d,g)$, $d\geq 3$ be a cohomogeneity one Riemannian manifold which is asymptotically flat, with outermost minimal surface boundary. If the scalar curvature is nonnegative $R\geq 0$, then inequality \eqref{ofinaoinoinoiqhhq} holds. Moreover, equality is achieved if and only if the manifold is isometric to the canonical slice of a Schwarzschild spacetime.
\end{theorem}

This article is organized as follows. In Section \ref{sec2} we describe cohomogeneity one initial data sets in detail, while in Section \ref{sec3} existence is established for the coupled Jang-IMCF system of equations in this setting. The proofs of Theorems \ref{foiqnoinoiqnhh}, \ref{foiqnoinoiqnhh1}, and \ref{thm3} are then given in Sections \ref{sec4}, \ref{sec5}, and \ref{sec6}. Finally, two appendices are provided with the first 
showing that linear momentum vanishes in the asymptotically flat case, and the second
exhibiting examples of relevant initial data.

\subsection*{Acknowledgements}
The authors would like to thank McKenzie Wang for helpful comments.

\section{Cohomogeneity One Initial Data Sets}
\label{sec2} \setcounter{equation}{0}
\setcounter{section}{2}

Consider a cohomogeneity one Riemannian manifold $(M^d,g)$ which is asymptotically flat or asymptotically hyperbolic, and possesses a smooth outermost apparent horizon boundary. If $\mathcal{G}$ is the compact connected Lie group of isometries giving rise to the cohomogeneity one structure, then as discussed in the introduction the manifold is diffeomorphic to a product $M^d\cong [0,\infty) \times \mathcal{G}/\mathcal{H}$, where $\mathcal{H}$ is a principal isotropy subgroup. To see this, first note that smooth outermost apparent horizons inherit the symmetries of the initial data from which they arise \cite[Lemma 3.1]{BKS} (\cite[Theorem 8.1]{AMS}, \cite{PS}), more precisely $\partial M^d$ is left invariant by $\mathcal{G}$. Therefore, we have a well-defined $\mathcal{G}$-action on $M^d \setminus\partial M^d$, and may then apply the principal orbit theorem to conclude that the orbit space $(M^d \setminus \partial M^d) / \mathcal{G}$ is diffeomorphic to a open half line $(0,\infty)$, with the origin corresponding to the apparent horizon, and thus $M^d \setminus \partial M^d \cong (0,\infty)\times \mathcal{G}/\mathcal{H}$. The desired conclusion now follows.

We next claim that in light of the asymptotics, the orbits $\mathcal{G}/\mathcal{H}$ must topologically be the sphere $S^{d-1}$. Note that if $s:M^d \rightarrow\mathbb{R}_+$ is the distance function to the boundary $\partial M^{d-1}$, then $g=ds^2 +g_s$ where $g_s$ is a 1-parameter family of $\mathcal{G}$-invariant metrics on $\mathcal{G}/\mathcal{H}$. Let $\Sigma_s$ denote the $s$-level set, then there exists a sufficiently large $s_0$ such that $\Sigma_{s_0}$ lies completely within the asymptotic end $M^d \setminus\mathcal{C}$. If not, then there exists a sequence of distances $s_i \rightarrow\infty$ with $|s_i -s_j|>1$ for $i\neq j$ such that $\Sigma_{s_i}\cap\mathcal{C}\neq \emptyset$. Since $\mathcal{C}$ is compact, a sequence of points $p_i \in \Sigma_{s_i}\cap\mathcal{C}$ must have a convergent subsequence. However this contradicts 
$|s_i -s_j|>1$, and we conclude that there is a $\Sigma_{s_0}\subset M^d\setminus \mathcal{C}$.
Because $\Sigma_{s_0}$ is compact, there exists a large coordinate sphere $S_{r_0}\subset\mathbb{R}^d$ whose preimage under the asymptotic diffeomorphism $\varphi$ separates $\Sigma_{s_0}$ from infinity. 
Let $\tilde{\Omega}_{s_0 r_0}$ be the image of $\Omega_{s_0 r_0}$ inside $\mathbb{R}^d \setminus B_1$, where $\Omega_{s_0 r_0}$ is the region bounded between $\Sigma_{s_0}$ and $\varphi^{-1}(S_{r_0})$.
By flowing along radial lines in $\mathbb{R}^d$, we may continuously deform $\tilde{\Omega}_{s_0 r_0}$ into $S_{r_0}$. This induces a deformation retract of $M_{s_0}^d$ onto $M_{r_0}^d$, where these are the closure of the noncompact component of $M^d \setminus \Sigma_{s_0}$ and $M^d \setminus \varphi^{-1}(S_{r_0})$, respectively. Since $M_{s_0}^d$ is homotopy equivalent to $\Sigma_{s_0}$ and $M_{r_0}^d$ is homotopy equivalent to $\varphi^{-1}(S_{r_0})$, it follows that
$\Sigma_{s_0}$ is a homotopy sphere. Therefore, by the resolution of the generalized Poincar\'{e} conjecture we have that $\mathcal{G}/\mathcal{H}=\Sigma_{s_0}$ is homeomorphic to the standard sphere $S^{d-1}$. Moreover, since exotic spheres do not admit smooth homogeneous space structures \cite[Theorem 1.1]{Schultz}, it follows that $\mathcal{G}/\mathcal{H}$ is diffeomorphic $S^{d-1}$.

The classification of homogeneous metrics on spheres has been given by Ziller \cite[page 352]{Ziller}. In all dimensions there is a unique $\mathrm{SO}(d)$-invariant metric (up to homothety) on $S^{d-1} = \mathrm{SO}(d)/ \mathrm{SO}(d-1)$, of constant curvature. This leads to spherically symmetric initial data, a setting in which the spacetime Penrose inequality has already been established. In odd dimensions the Hopf fibrations $S^1 \hookrightarrow S^{2n+1} \to \mathbb{CP}^n$ give rise to a 1-parameter family or 2-parameter family of $\mathrm{SU}(n+1)$-invariant metrics on $S^{2n+1}=\mathrm{SU}(n+1)/\mathrm{SU}(n)$ depending on whether $n>1$ or $n=1$, respectively. For $n>1$, these metrics have an extra $\mathrm{U}(1)$ symmetry, and will be discussed in detail below. The remaining cases include a 3-parameter family of metrics on 
$S^{4n+3} = \mathrm{Sp}(n+1) / \mathrm{Sp}(n)$ and a 1-parameter family of metrics on  $S^{15} = \mathrm{Spin}(9) / \mathrm{Spin}(7)$ up to homothety, which will be examined further in Section \ref{sec6}.

In what follows, an initial data set $(M^d,g,k)$ will be referred to as $\mathcal{G}$-\textit{invariant} if the Riemannian manifold $(M^d,g)$ is cohomogeneity one with respect to $\mathcal{G}$, and the Lie derivative vanishes $L_{\eta}k=0$ for any Killing field $\eta$ associated with the isometries of $g$.

\subsection{$\mathrm{SU}(n+1)$-invariant initial data, $n>1$}

The 1-parameter family of Berger metrics on the `squashed' sphere $S^{2n+1}$ are given by
\begin{equation}\label{Berger}
g_{\lambda} = \lambda (d \psi + A)^2 + g_{FS},
\end{equation} 
where $\psi$ is a $2\pi$-periodic coordinate on the circles of the Hopf fibration, $g_{FS}$ is the Fubini-Study metric on $\mathbb{CP}^n$ scaled so that $\text{Ric}(g_{FS}) = 2(n+1) g_{FS}$, and $A$ is a connection 1-form such that $\omega = dA/2$ is the associated K\"ahler form. The normalization is chosen so that $g_1$ is the unit round metric on $S^{2n+1}$. Note that the squashing parameter $\lambda$ controls the size of the $S^1$ fibers in the fibration, and the metric
\eqref{Berger} is invariant under the $\mathrm{U}(1)$ isometry generated by the Killing field $\partial_\psi$. Moreover, since the round metric inherits the subgroup of isometries $\mathrm{SU}(n+1)$ from $\mathbb{C}^{n+1}$ and the action descends by isometries to the complex projective space quotient, we find that $L_\eta (d\psi+A)=0$ for any Killing field $\eta$ associated with the $\mathrm{SU}(n+1)$ symmetry. It follows that the Berger metrics inherit this symmetry, and they are the building blocks of $\mathrm{SU}(n+1)$-invariant initial data.

\begin{prop}\label{aofnoaigopinjopinhipoq}
Let $n>1$. Consider an asymptotically flat or asymptotically hyperbolic $\mathrm{SU}(n+1)$-invariant initial data set $(M^{2(n+1)},g,k)$, with outermost apparent horizon boundary. Then $M^{2(n+1)}\cong [0,\infty) \times S^{2n+1}$ and the metric and extrinsic curvature take the form
\begin{align}\label{aonfoih}
\begin{split}
g =& ds^2 + \rho(s)^2 \left[ e^{-4nB(s)} (d\psi  + A)^2 + e^{2B(s)} g_{FS} \right], \\
k =& k_a ds^2 + 2 k_s \rho e^{-2nB} ds (d\psi + A) + \rho^2 \left[ k_b e^{-4nB} (d\psi  + A)^2 + k_c e^{2B} g_{FS} \right],
\end{split}
\end{align} 
for some smooth functions $B$, $k_a$, $k_b$, $k_c$, $k_s$, and $\rho>0$ of $s$ alone. In an asymptotically flat end
\begin{equation}\label{fnoanoih}
\rho(s)=s +O_2(s^{1-\tau}), \quad B=O_2(s^{-\tau}), \quad k_a, k_b, k_c, k_s =O_1(s^{-\tau-1}),\quad \mathrm{Tr}_g k=O_1(s^{-2\tau-1}),
\end{equation}
while in an asymptotically hyperbolic end 
\begin{equation}\label{asymptoticallyhyperbolic}
\rho(s)=\sinh s +O_2(e^{(1-q)s}), \quad B=O_2(e^{-qs}), \quad k_a, k_b, k_c, k_s =O_1(e^{-qs}).
\end{equation}
\end{prop}

\begin{proof}
According to the discussion at the beginning of this section, and the classification of \cite{Ziller}, the manifold is topologically $M^{2(n+1)}=[0,\infty)\times S^{2n+1}$ and
the metric is given by
\begin{equation}\label{GGmetric}
g = ds^2 + P(s)^2 (d\psi + A)^2 + Q(s)^2 g_{FS}
\end{equation}
for some smooth positive functions $P(s)$ and $Q(s)$. We may then set
\begin{equation}
\rho(s): = \left[P(s) Q(s)^{2n} \right]^{\frac{1}{2n+1}}, \qquad e^{B(s)}:=\left(\frac{Q(s)}{P(s)}\right)^{\frac{1}{2n+1}},
\end{equation} 
so that
\begin{equation}
Q(s) = e^{B(s)} \rho(s), \qquad P(s) = e^{-2nB(s)} \rho(s).
\end{equation} 
The desired expression for the metric now follows.

Consider now the structure of $k$. It is useful to introduce the orthonormal coframe
\begin{equation}\label{frame}
\mathbf{e}^1 = ds, \qquad \mathbf{e}^2 = \rho(s) e^{-2n B(s)}(d\psi + A), \qquad \mathbf{e}^i = \rho(s) e^{B(s)} \hat{e}^i, \quad i=3,\ldots,2(n+1),
\end{equation} 
where $\hat{e}^i$ are members of an orthornoaml coframe for the Fubini-Study metric. The dual basis vectors are given by
\begin{equation}\label{foanofinh}
\mathbf{e}_1 = \partial_s, \qquad \mathbf{e}_2 = \rho(s)^{-1}e^{2nB(s)} \partial_\psi, \qquad \mathbf{e}_i = \rho(s)^{-1} e^{-B(s)} \left(\hat{e}_i - A(\hat{e}_i) \partial_\psi \right).
\end{equation} 
We may then expand the extrinsic curvature as $k = k_{\mathrm{ij}}\mathbf{e}^{\mathrm{i}} \otimes \mathbf{e}^{\mathrm{j}}$ for some symmetric matrix of functions $k_{\mathrm{ij}}$, with $\mathrm{i},\mathrm{j}=1,\ldots,2(n+1)$. First observe that $k_{\mathrm{ij}}$ has no dependence on $\psi$, since the Killing field $\partial_{\psi}$ commutes with the frame \eqref{foanofinh}. Next, note that $g_{FS}$ is the unique $\mathrm{SU}(n+1)$-invariant metric (up to a scaling) on $\mathbb{CP}^n$.  This implies that there can be no nontrivial $\mathrm{SU}(n+1)$-invariant 1-forms $\sigma$ on $\mathbb{CP}^n$, for otherwise $g_{FS} + \lambda \sigma^2$ would produce another invariant family of metrics.  It follows that an invariant $k$ cannot contain terms of the form $\mathbf{e}^1 \otimes \mathbf{e}^i$ or $\mathbf{e}^2 \otimes \mathbf{e}^i$ for any $i>2$. To see this, write $k_{\ell i}\mathbf{e}^{1}\otimes \mathbf{e}^{i}=\mathbf{e}^{\ell}\otimes K_\ell$
where $K_\ell=k_{\ell i}\mathbf{e}^{i}$ and $\ell=1,2$. Let $\eta$ be a Killing field associated with the $\mathrm{SU}(n+1)$ symmetry, then $L_{\eta}(\mathbf{e}^{\ell}\otimes K_\ell)=\mathbf{e}^{\ell}\otimes L_{\eta}K_{\ell}$. However, since $L_{\eta}K_{\ell}$ is a linear combination of
$\mathbf{e}^{i}$ for $i>2$, there is no possibility of canceling this term with any other expression in $L_{\eta}k$. We then have that $L_{\eta}K_{\ell}=0$. Thus, $K_{\ell}$ is a $\mathrm{SU}(n+1)$-invariant 1-form as $\eta$ was arbitrary, so that $K_{\ell}=0$.  From this, one may further conclude that $k_{11}$, $k_{12}$, and $k_{22}$ depend only on $s$, and that $\alpha=\sum_{i,j>2}k_{ij}e^i \otimes e^j$ is $\mathrm{SU}(n+1)$-invariant. The latter statement implies that $\alpha$ must be a multiple of $g_{FS}$, otherwise for small $\varepsilon$ the expression $g_{FS}+\varepsilon \alpha$ would furnish an invariant metric on $\mathbb{CP}^n$ in violation of the uniqueness property of the Fubini-Study metric.  Putting this all together yields the desired structure for $k$.

Lastly, to obtain the fall-off in \eqref{fnoanoih}, write the Euclidean metric of \eqref{1.3} in polar form with the coordinate spheres expressed with respect to the Hopf fibration, and compare  with the coefficients of $\varphi_{*}g$. This gives relations between the two sets of coordinates and the quantities $\rho$, $B$, from which the desired decay may be derived. An analogous procedure applies to the extrinsic curvature. Moreover, the asymptotically hyperbolic case is treated similarly.
\end{proof}

It will be useful to record some facts about the class of metrics \eqref{aonfoih}. First note that the $s$-level set spheres $\Sigma_s$ have volume form
\begin{equation}\label{volumeform}
dV_\Sigma=\rho^{2n+1} d\psi \wedge dV_{\mathbb{CP}^n} = \rho^{2n + 1} dV_{S^{2n+1}},
\end{equation} 
and scalar curvature
\begin{equation}\label{rsigma}
R_{\Sigma} =  - \frac{2n P^2}{Q^4} + \frac{4n(n+1)}{Q^2} = -\frac{2n}{\rho^2} \left(e^{-4(n+1)B} - 2(n+1)e^{-2B} \right),
\end{equation} 
while the scalar curvature of $g$ as expressed in \eqref{GGmetric} is given by
\begin{equation}
R = -\frac{4n Q''}{Q} - \frac{2 P''}{P} - \frac{2n P^2}{Q^4} - \frac{4n P' Q'}{PQ} + \frac{4n(n+1)}{Q^2} - \frac{2n(2n-1) Q'^2}{Q^2}
\end{equation}  
where the prime notation indicates differentiation with respect to $s$. 
Furthermore
\begin{equation}\label{tracek12}
\mathrm{Tr}_g k = k_a + k_b + 2n k_c, \qquad \mathrm{Tr}_{\Sigma}k = k_b + 2n k_c,
\end{equation} 
and the mean curvature of the level sets is
\begin{equation}
H_\Sigma = (2n+1) \frac{\rho'}{\rho}.
\end{equation}

\subsection{$\mathrm{SU}(2)$-invariant initial data}\label{section2.2}

For $n=1$, $S^3$ is diffeomorphic to the Lie group $\mathrm{SU}(2)$ and hence there exist homogeneous metrics invariant under the left or right action of $\mathrm{SU}(2)$ that do not possess the enhanced $\mathrm{U}(1)$ isometry.  Consider the following right-invariant 1-forms \cite[Appendix A]{Lucietti}
\begin{equation}
\sigma^1 = \sin\psi d \theta - \cos\psi \sin\theta d \phi, \quad \sigma^2 = \cos\psi d \theta + \sin\psi \sin\theta d \phi, \quad \sigma^3 = d \psi + \cos\theta d \phi,
\end{equation} 
where $\theta \in (0,\pi), \phi \in (0,2\pi)$, and $\psi \in (0,4\pi)$. These satisfy
$d\sigma^i = -\tfrac{1}{2} \epsilon_{ijk} \sigma^j \wedge \sigma^k$, where $\epsilon_{ijk}$ is totally antisymmetric and $\epsilon_{123}=1$. We then have the following $\mathrm{SU}(2)$-invariant metrics on the 3-sphere $g_{SU}=h_{ij} \sigma^i \sigma^j $, in which
$h$ is a constant positive definite matrix. Note that $h_{ij}=\tfrac{1}{4}\delta_{ij}$ gives rise to the unit round metric. Furthermore, an orthogonal frame may be obtained with a diagonalizing $O\in\mathrm{SO}(3)$ by setting $\hat{\sigma}^i \to O^i_{j} \sigma^j$. It follows that
\begin{equation}\label{diagSU(2)}
g_{SU} = \frac{1}{4} \left(c_1^2 (\hat{\sigma}^1)^2 + c_2^2 (\hat{\sigma}^2)^2 + c_3^2 (\hat{\sigma}^3)^2 \right),
\end{equation} 
where $c_i > 0$ are positive constants. The Berger class of $\mathrm{SU}(2) \times \mathrm{U}(1)$-invariant metrics is obtained by setting $c_1 = c_2$; in this case the additional Killing field is given by $\partial_\psi$.  A computation shows that the volume form is $dV_{g_{SU}} = c_1 c_2 c_3 dV_{g_{S^3}}$, and the scalar curvature is given by
\begin{equation}\label{suscalarcurv}
R_{SU} = \frac{2}{c_1^2 c_2^2 c_3^2} \left(2 c_1^2 (c_2^2 + c_3^2) - c_1^4 - (c_2^2 - c_3^2)^2 \right).
\end{equation} 

Consider now $\mathrm{SU}(2)$-invariant initial data $(M^4,g,k)$ satisfying the hypotheses of  Proposition \ref{aofnoaigopinjopinhipoq}. Then $M^4 \cong [0,\infty) \times S^3$, the metric takes the form
\begin{equation}
g = ds^2 + \frac{1}{4} \left(c_1^2(s) (\hat{\sigma}^1)^2 + c_2^2(s) (\hat{\sigma}^2)^2 + c_3^2(s) (\hat{\sigma}^3)^2 \right)
\end{equation}
where $\hat{\sigma}_i$ also depends on $s$ but only through the orthogonal transformation $O$, and the extrinsic curvature may be expressed as
\begin{equation}
k = k_a(s) ds^2 + 2 \hat{k}_i(s) e^i ds + \hat{k}_{ij}(s) e^i e^j
\end{equation} 
for some smooth functions $k_a$, $\hat{k}_i$, and $\hat{k}_{ij}$ of $s$ alone, where $e^i(s)=(c_i(s)/2)\hat{\sigma}^i$ is an orthonormal frame for each coordinate sphere. Moreover, as in the proof of Proposition \ref{aofnoaigopinjopinhipoq}, it may be shown that in an asymptotically flat end
\begin{equation}\label{fnoanoih222}
c_i(s)=s +O_2(s^{1-\tau}), \quad\quad\quad k_a(s), \hat{k}_i(s), \hat{k}_{ij}(s) =O_1(s^{-\tau-1}),
\end{equation}
while in an asymptotically hyperbolic end 
\begin{equation}\label{asymptoticallyhyperbolic222}
c_i(s)=\sinh s +O_2(e^{(1-q)s}), \quad\quad \quad  k_a(s), \hat{k}_i(s), \hat{k}_{ij}(s) =O_1(e^{-qs}).
\end{equation}


\section{The Cohomogeneity One Generalized Jang Equation}
\label{sec3} \setcounter{equation}{0}
\setcounter{section}{3}

The goal of this section is to obtain existence and asymptotics for solutions of a coupled system involving inverse mean curvature flow and the generalized Jang equation, in the setting of $\mathrm{SU}(n+1)$-invariant initial data sets $(M^{2(n+1)},g,k)$. Recall that the generalized Jang equation \cites{Bray:2009oni,Bray:2009au} associated with the data set is given by 
\begin{equation}
\left(g^{ij} - \frac{\phi^2 f^i f^j}{ 1 + \phi^2 |\nabla f|^2}\right)\left(\frac{\phi \nabla_{ij} f + \phi_i f_j +\phi_j f_i}{\sqrt{1 + \phi^2 |\nabla f|^2}} - k_{ij} \right) = 0,
\end{equation} 
where $f^{i}=g^{ij}f_j$, $f_j=\partial_j f$, and $\nabla$ is the Levi-Civita connection. Utilizing the structure of the metric and extrinsic curvature from Proposition \ref{aofnoaigopinjopinhipoq}, and assuming that $\phi$, $f$ are functions of $s$ alone, we find that this equation is equivalent to 
\begin{equation}\label{fobqoboibqihq}
v' + \frac{(2n+1) \rho'}{\rho}v +   \left(\frac{\phi'}{\phi} v -k_a\right)(1-v^2) - k_b - 2n k_c =0,
\end{equation} 
where prime indicates differentiation with respect to $s$ and
\begin{equation}
v= \frac{\phi f'}{\sqrt{1 + \phi^2 f'^2}}.
\end{equation}
Since inverse mean curvature flow emanating from a $\mathrm{SU}(n+1)$-invariant apparent horizon remains $SU(n+1)$-invariant, following the proposal in \cites{Bray:2009oni,Bray:2009au} to couple the generalized Jang equation with inverse mean curvature flow leads to
\begin{equation}\label{fouqoboiqh}
\phi = \frac{d\rho}{d\bar{s}}=\sqrt{1-v^2}\rho'=\left(\frac{|\Sigma_{\bar{s}}|}{\omega_{2n+1}}\right)^{\frac{1}{2n+1}}\frac{\bar{H}}{2n+1},
\end{equation}
where $\Sigma_{\bar{s}}$ is the $\bar{s}$-level set surface with mean curvature $\bar{H}$ with respect to the Jang metric $\bar{g}=g+\phi^2 df^2$, $|\Sigma_{\bar{s}}|$ is the area, and $\bar{s}$ is the corresponding radial arclength parameter 
\begin{equation}
\bar{s}= \int_0^s \sqrt{1 + \phi^2 f'^2} = \int_0^s \frac{1}{ \sqrt{1 - v^2}}.
\end{equation} 
This choice of $\phi$ ensures monotonicity of the Hawking mass along inverse mean curvature flow in the  Jang manifold $(M^{2(n+1)},\bar{g})$.

The boundary of the initial data is assumed to be an outermost apparent horizon, and therefore we must have
\begin{equation}
0<\theta_{\pm}=H \pm \text{Tr}_{\Sigma_s} k = (2n+1) \frac{ \rho'}{\rho} \pm \left(k_b + 2n k_c\right)
\end{equation}
for $s>0$, and $\theta_+ (0) =0$ and/or $\theta_-(0) =0$ depending on whether $\Sigma_0=\partial M^{2(n+1)}$ is a future and/or past horizon. It follows that the $g$-mean curvature of radial surfaces is positive $H=\tfrac{1}{2}(\theta_+ +\theta_-)>0$ away from the boundary, and hence on this region $\phi>0$ as long as $f'$ is bounded. Inserting the expression \eqref{fouqoboiqh} for $\phi$ into the generalized Jang equation \eqref{fobqoboibqihq} and rearranging terms produces
\begin{equation}\label{foqoifnoiwhg}
(1-v^2) v' + (1-v^2)F_\mp (s,v)  \pm \theta_\mp = 0,
\end{equation} 
where
\begin{equation}
F_\mp (s,v):= \mp \frac{(2n+1)}{1 \pm v} \frac{\rho'}{\rho} + \frac{ v \rho''}{\rho'} - k_a. 
\end{equation}

\begin{remark} 
The above computations have been performed for $\mathrm{SU}(n+1)$-invariant initial data with $n > 1$. In particular, Proposition \ref{aofnoaigopinjopinhipoq} has been utilized. For the special case when $n=1$, we may appeal to Section 2.2 to obtain an analogous equation \eqref{foqoifnoiwhg}. 
In what follows we will treat the two cases simultaneously without noting differences when $n=1$, 
as the modifications needed for this dimension are straightforward.
\end{remark}

\begin{theorem}\label{jangexist}
Let $(M^{2(n+1)},g,k)$, $n\geq 1$ be an asymptotically flat or asymptotically hyperbolic $\mathrm{SU}(n+1)$-invariant initial data set, with outermost apparent horizon boundary.
If the boundary mean curvature satisfies $H(0)\neq 0$, then given $\alpha\in(-1,1)$ there
exists a unique solution $v\in C^{1}([0,\infty))\cap C^{\infty}((0,\infty))$ of \eqref{foqoifnoiwhg} such that $-1<v(s)<1$ for $s>0$ and $v(0)=\alpha$. This remains true when $\alpha= \pm 1$ if $\partial M^{2(n+1)}$ is a past (future) apparent horizon, respectively. If $H(0)=0$, then the same existence statement holds with $v(0)=0$ and $v\in C^{0}([0,\infty))\cap C^{\infty}((0,\infty))$. Furthermore, in all cases the solution admits the following decay 
\begin{equation}
|v(s)|+s|v'(s)|\leq Cs^{-2\tau_0}
\end{equation}
in the asymptotically flat setting for any $\tau_0 <\min\{\tau,n+\tfrac{1}{2}\}$, and the decay
\begin{equation}\label{vfall}
|v(s)|+|v'(s)|\leq Ce^{-q_0 s}
\end{equation}
in the asymptotically hyperbolic setting for any $q_0<\min\{q,2n+2\}$, where $C$ is a constant depending on $g$ (up to second derivatives) and $k$ as well as on $\tau_0$ or $q_0$, respectively.
\end{theorem}

\begin{proof}
Assume first that the initial data are asymptotically flat.
Consider the case when $v(0)=\alpha$ with $|\alpha|<1$ and $H(0)\neq 0$. By standard methods there exists a $C^1$
solution on some maximal interval $[0,\bar{s})$ of nonzero length. We claim that the solution must satisfy the basic uniform pointwise bound 
\begin{equation}\label{aofnoainh}
|v(s)|<1 \quad\quad  \text{ for all } s.
\end{equation}
To see this, proceed by contradiction and assume that the estimate fails, then there exists a first $s_0>0$ such that $v(s_0)=\pm 1$. Without loss of generality we may take $v(s_0)=1$, then there is an increasing sequence $s_i \rightarrow s_0$ with $v'(s_i)\geq 0$ and $|v(s_i)|<1$. It follows that
\begin{equation}
0\leq \lim_{i\rightarrow \infty} (1-v^{2})v'(s_i)=\lim_{i\rightarrow \infty}\left[-\theta_{-}(s_i)
-(1-v^{2})F_{-}(s_i,v(s_i))\right]=-\theta_{-}(s_0)<0,
\end{equation}
which yields a contradiction. We conclude that \eqref{aofnoainh} is valid. In fact, this estimate may be improved. 
Observe that since $H(0)>0$, with the help of Proposition \ref{aofnoaigopinjopinhipoq} and the discussion preceding this theorem, we have that $\rho'(s)\geq c$ globally for some constant $c>0$. Thus, there exists a positive radial function $\bar{F}_{-}$ satisfying
\begin{equation}\label{93hrqhoihgiqh}
|F_{-}(v,s)|\leq \bar{F}_{-}(s)\leq \frac{c_1}{1+s}\quad \text{ for all }s\in[0,\infty) \text{ whenever }v(s)\geq 0,
\end{equation}
for some constant $c_1$. If $\partial M^{2(n+1)}$ is not a past apparent horizon, that is $\theta_-(0)>0$, then there exists a constant $\delta>0$ such that $\bar{F}_{-}(s)^{-1}\theta_{-}(s)> \delta$; the constant may also be chosen small enough to satisfy $v(0)=\alpha<\sqrt{1-\delta}$.
If $s_1$ is the first point at which $v(s_1)=\sqrt{1-\delta}$ then $v'(s_1)\geq 0$ and consequently
\begin{equation}\label{oqoqihoihnw}
0\leq -\theta_{-}(s_1)+(1-v^{2}(s_1))\bar{F}_{-}(s_1),
\end{equation}
which implies that $1-v^2(s_1)>\delta$ yielding a contradiction.
We conclude that $v(s)<\sqrt{1-\delta}$ at all points, and a similar argument with $F_-$ replaced by $F_+$ provides a lower bound when $\partial M^{2(n+1)}$ is not a future apparent horizon. If $\partial M^{2(n+1)}$ is a past apparent horizon, that is $\theta_-(0)=0$, then since $\theta_-\geq 0$ we have
\begin{equation}\label{ofhqnoifnoianiohgqah}
v'(s)=-\theta_-(s)(1-v^2(s))^{-1} - F_-(s,v(s))\leq \bar{F}_-(s)\leq c_1
\end{equation}
whenever $v$ is nonnegative.
It follows that $v(s)\leq c_1 s+\alpha$ for $s\in [0,\bar{s}_1)$ where $\bar{s}_1=\min\{\bar{s},c_{1}^{-1}(1-\alpha)\}$.
Since $\theta_-$ is strictly positive at $\bar{s}_1 /2$, by applying the same arguments in the non-past apparent horizon case to the region $s\geq \bar{s}_1 /2$, we find that globally $v(s)<\sqrt{1-\delta}$ for some $\delta>0$ depending on $\alpha$. As before, analogous arguments with $F_-$ replaced by $F_+$ yield a lower bound when $\partial M^{2(n+1)}$ is a future apparent horizon. Therefore, in all cases 
\begin{equation}\label{o0ihj9iqjq}
|v(s)|<\sqrt{1-\delta}
\end{equation}
on the domain of existence. These pointwise estimates imply uniform control on first derivatives immediately from equation \eqref{foqoifnoiwhg}. The function $v$ is then uniformly continuous, and thus has a unique and continuous extension to $[0,\bar{s}]$. By differentiating the equation, the same holds true for all derivatives of $v$. Hence the solution has an unrestricted domain, that is $v\in C^{\infty}([0,\infty))$.

Consider now the case when $\alpha=1$, $\partial M^{2(n+1)}$ is a past apparent horizon, and $H(0)\neq 0$. 
Observe that for $-\tfrac{3}{2}\leq w\leq \tfrac{1}{2}$ or equivalently $|v|\leq 1$ equation \eqref{foqoifnoiwhg} may be expressed as
\begin{equation}\label{ogihqoihnogiwnh}
(1+v(w))w'=-\theta_- +F_-(s,v(w))\left(1-2w -2\sqrt{1-2w}\right),
\end{equation}
where $w=v-\frac{v^2}{2}$ and $v(w)=1-\sqrt{1-2w}$.
Notice that \eqref{ogihqoihnogiwnh} is no longer degenerate in that the principal symbol does not vanish. To produce a solution 
in the required range, we may take a limit of solutions $w_{\varepsilon}$ with initial condition $w_{\varepsilon}(0)=\tfrac{1}{2}-\varepsilon$. Due to the nondegeneracy of this equation and the estimate \eqref{o0ihj9iqjq}, the solution $w:=\lim_{\varepsilon\rightarrow 0}w_{\varepsilon}$ is $C^1$ up to the boundary, and satisfies $-\tfrac{3}{2} \leq w\leq \tfrac{1}{2}$. In fact, since $\theta_-(s)>0$ for $s>0$ we must have that $w(s)<\tfrac{1}{2}$ for $s>0$; similarly, replacing the $F_-$ with $F_+$ as before shows that also $w(s)>-\tfrac{3}{2}$. It may further be shown that the corresponding $v$ has continuous first derivatives up to the boundary. To see this rewrite equation \eqref{foqoifnoiwhg} once more as
\begin{equation}\label{fonoinqoinqh234}
\left(u+\mathcal{F}_-\right)u'=2v\theta_-,
\end{equation}
where $u=1-v^2-\mathcal{F}_-$ and $\mathcal{F}_-=\int_{0}^{s}2vF_-$. Integrating \eqref{fonoinqoinqh234} and then dividing by $s^2$ yields a quadratic equation for $u/s$, with coefficients that have finite limits as $s\rightarrow 0$. This implies existence of the first derivative for $v$ at $s=0$, which may be found as the nonpositive solution of
\begin{equation}
2v'(0)^2+2F_-(0)v'(0)-\theta'_-(0)=0.
\end{equation}
It follows that $v\in C^{1}([0,\infty))\cap C^{\infty}((0,\infty))$ satisfies \eqref{foqoifnoiwhg} with $v(0)=1$. An analogous argument shows that one may solve the equation with $v(0)=-1$ when $\partial M^{2(n+1)}$ is a future apparent horizon. 

Lastly, consider the case when $H(0)=0$. The primary difference in this situation concerns the functions $F_{\mp}$, namely they must blow-up at the boundary due to the term $v\tfrac{\rho''}{\rho'}$, unless $v(0)=0$. Thus, in this case we will only consider the initial condition $v(0)=0$. Note that with this condition, the principal part of the equation is nondegenerate near $s=0$, however there is a singular behavior in the coefficients:
\begin{equation}\label{ifihwoihihihwihgiw}
v'+\frac{\rho''}{\rho'}v=k_a +\frac{(2n+1)\rho'}{(1+v)\rho}-\frac{\theta_-}{1-v^2}.
\end{equation}
Approximate solutions $v_{\varepsilon}\in C^{\infty}([\varepsilon,1))$, with $v_{\varepsilon}(\varepsilon)=0$ and $|v_{\varepsilon}|<1$, may be obtained with uniform estimates on compact subsets of $(0,1)$ as $\varepsilon\rightarrow 0$. Thus, a diagonal argument ensures subconvergence to a solution $v\in C^{\infty}((0,1))$. Using the inverse of the linear operator on the left-hand side of \eqref{ifihwoihihihwihgiw}, we find that the limit function may be represented by $v(s) =\rho'(s)^{-1}\int_{0}^s \rho'u$ where $u$ is uniformly bounded. It follows that the solution is continuous up to the boundary with $v(0)=0$, and this may then be extended globally using the methods above, so that $v\in C^{0}([0,\infty))\cap C^{\infty}((0,\infty))$.

Finally, it will be shown that $v$ has the correct asymptotics at
infinity. Recall that we have already assumed the asymptotically flat hypothesis. By the fall-off conditions \eqref{fnoanoih} and the estimate \eqref{o0ihj9iqjq}, we may write equation \eqref{foqoifnoiwhg} as
\begin{equation}\label{fpoihqoignoiqh}
v'+\frac{(2n+1)s^{-1}}{1-v^{2}}v=O(s^{-2\tau-1}+s^{-\tau-1}|v|),\quad\quad s\geq s_0 >0.
\end{equation}
The solution on this exterior region can then be represented by
\begin{equation}
v(s)\!=\!\exp\left(\!\!-\!\!\int_{s_{0}}^{s}\!\!\frac{(2n+1)r^{-1}}{1-v^{2}}dr\!\right)
\!\!\left[\int_{s_{0}}^{s}\!\!\!O(r^{-2\tau-1}\!\!+\! r^{-\tau-1}|v|)\exp\left(\int_{s_{0}}^{r}\!\frac{(2n+1)t^{-1}}{1-v^{2}}dt\right)dr
\!+\!v(s_{0})\right].
\end{equation}
From this it follows that $v$ decays to zero polynomially. To estimate the rate, choose any positive $\tau_0< \min\{\tau,n+\tfrac{1}{2}\}$ and denote the exponential function inside the integral as $\mathcal{E}(r)$.
Note that $r^{-2\tau_0}\mathcal{E}(r)$ is increasing for $r\geq s_0$ if $s_0$ is large enough. Then extracting this function from the integral at its maximum value shows that the term involving $O(r^{-2\tau-1})$ decays on the order of $s^{-2\tau_0}$; clearly also the last term involving $v(s_0)$ decays at this rate as well. For the term involving $|v|$, extracting $r^{-\tau_0}\mathcal{E}(r)$ in the same way yields decay for $v$ on the order of $s^{-\tau_0}$. Inserting this estimate back into the same term and extracting $r^{-2\tau_0}\mathcal{E}(r)$ produces the desired decay, namely, with the help of \eqref{fpoihqoignoiqh} for the derivative fall-off we find that
\begin{equation}
|v(s)|+s|v'(s)|\leq Cs^{-2\tau_0}
\end{equation}
for some constant $C$.

Assume now that the initial data are asymptotically hyperbolic. Only a slight modification of the above arguments is needed to obtain global existence of a solution satisfying the required initial conditions. More precisely, in this case $F_- =-(2n+1)+O(e^{-2s})$ and $\theta_-=2n+1+O(e^{-2s})$ in the asymptotic end, and therefore the function $\bar{F}_-(s)$ used in \eqref{93hrqhoihgiqh}-\eqref{ofhqnoifnoianiohgqah}
should be taken instead to be a large constant. Consider now the asymptotics of the solution, and with the help of \eqref{asymptoticallyhyperbolic} rewrite equation \eqref{foqoifnoiwhg} as
\begin{equation}
v'+\left(\frac{(2n+1)\rho'}{(1-v^2)\rho}+\frac{\rho''}{\rho'}\right) v=O(e^{-qs}+e^{-qs}|v|), \quad\quad s\geq s_0 >0.
\end{equation}
The solution on this exterior region can then be represented by
\begin{equation}
v(s)=e^{-\int_{s_0}^s \mathbf{A}(r)dr}\left(\int_{s_0}^s O(e^{-qr}+e^{-qr}|v|)e^{\int_{s_0}^{r}\mathbf{A}(t)dt}dr +v(s_0)\right),\quad\text{ } \mathbf{A}(s)=\frac{(2n+1)\rho'}{(1-v^2)\rho}+\frac{\rho''}{\rho'}.
\end{equation}
From this it follows that $v$ decays to zero exponentially. Since $\mathbf{A}(s)=2n+2 +o(1)$, an argument analogous to that in the asymptotically flat case may be used to obtain
\begin{equation}
|v(s)|+|v'(s)|\leq Ce^{-q_0 s},
\end{equation}
for any $q_0<\min\{q,2n+2\}$.
\end{proof}

\section{Proof of Theorem \ref{foiqnoinoiqnhh}: Asymptotically Flat Penrose Inequality}
\label{sec4} \setcounter{equation}{0}
\setcounter{section}{4}

Here we will establish the spacetime Penrose inequality for asymptotically flat $\mathrm{SU}(n+1)$-invariant initial data sets. This will be accomplished with an averaged monotonicity for the Hawking mass along inverse mean curvature flow within the Jang deformed manifold. A crucial observation in this setting, which replaces the use of Gauss-Bonnet, is that the scalar curvature of the leaves cannot be `too large'.
Let $\Sigma\subset (M^d,g)$ be a smooth hypersurface in a $d$-dimensional Riemannian manifold, then the Hawking mass of $\Sigma$ is given by
\begin{equation}
\mathbf{m}_H(\Sigma) := \frac{1}{2} \left(\frac{|\Sigma|}{\omega_{d-1}}\right)^{\frac{d-2}{d-1}}\left[ 1 - \frac{1}{(d-1)^2 \omega_{d-1}^{\frac{2}{d-1}} |\Sigma|^{\frac{d-3}{d-1}} }\int_\Sigma H^2 dV\right],
\end{equation}
where $H$ is the mean curvature of $\Sigma$ and $|\Sigma|$ denotes its area. If $\{\Sigma_t\}_{t=0}^{\infty}$ denotes a smooth inverse mean curvature flow, then a direct computation (see e.g. \cite[Theorem 4.27]{Leetext} for the case $d=3$) shows that
\begin{align} \label{dmdt}
\begin{split}
\frac{d}{dt} \mathbf{m}_H (\Sigma_t) =& \frac{1}{2} \left( \frac{|\Sigma_t|}{\omega_{d-1}}\right)^{\frac{d-2}{d-1}} \left( \frac{d-2}{d-1} - \frac{1}{(d-1)^2 \omega_{d-1}^{\frac{2}{d-1}} |\Sigma_t|^{\frac{d-3}{d-1}}} \int_{\Sigma_t} R_{\Sigma} dV\right) \\
+&  \frac{1}{2} \left( \frac{|\Sigma_t|}{\omega_{d-1}}\right)^{\frac{d-2}{d-1}} \! \frac{1}{(d-1)^2 \omega_{d-1}^{\frac{2}{d-1}} |\Sigma_t|^{\frac{d-3}{d-1}}} \int_{\Sigma_t} \left( \frac{|\nabla_\Sigma H|^2}{H^2} + |II|^2 - \frac{H^2}{d-1} + R\right)dV,
\end{split}
\end{align}
where $R_{\Sigma}$ is the scalar curvature and $II$ is the second fundamental form of $\Sigma_t$.

\begin{prop}\label{prop:monotonicity}
Let $(M^{2(n+1)},g)$ be an asymptotically flat or asymptotically hyperbolic $\mathrm{SU}(n+1)$-invariant Riemannian manifold with outermost minimal surface boundary, for $n\geq 1$. If $\Sigma$ is a level set of the distance function to the boundary then
\begin{equation}
\int_{\Sigma}R_{\Sigma}dV\leq c_n |\Sigma|^{\frac{2n-1}{2n+1}},
\end{equation}
where $c_n = 2n(2n+1)\omega_{2n+1}^{\frac{2}{2n+1}}$. Moreover, equality is achieved if and only if $\Sigma$ is a round sphere.
\end{prop}

\begin{proof}
Consider first the case $n>1$. According to Proposition \ref{aofnoaigopinjopinhipoq}, \eqref{volumeform}, and \eqref{rsigma} we have
\begin{equation}\label{monotonicity}
 2n+1 - \frac{1}{2n\omega_{2n+1}^{\frac{2}{2n+1}} |\Sigma|^{\frac{2n-1}{2n+1}}} \int_\Sigma R_\Sigma dV=
2n + 1 + e^{-4(n+1)B(s)} -2(n+1)e^{-2B(s)}.
\end{equation}
Setting $\varrho= e^{-2B}$ yields the polynomial
\begin{equation}\label{aofnoiqanoinpoih}
I_n(\varrho) = 2n+1 + \varrho^{2(n+1)} - 2(n+1) \varrho.
\end{equation}
This function satisfies $I_n(0)=2n+1$, it monotonically decreases on $(0,1)$ to a minimum $I_n(1) = 0$, and then increases monotonically for $\varrho\geq 1$. It follows that \eqref{monotonicity} is always nonnegative, and vanishes only when $B=0$ which coincides with a round $\Sigma$.

Consider now the case $n=1$. According to Section \ref{section2.2} we have
\begin{equation}\label{monotonicity1}
 3 - \frac{1}{2\omega_{3}^{2/3} |\Sigma|^{1/3}} \int_\Sigma R_\Sigma dV=
3-\frac{1}{(c_1 c_2 c_3)^{4/3}}\left[2c_1^2(c_2^2 +c_3^2)-c_1^4 -(c_2^2 -c_3^2)^2 \right].
\end{equation}
Since this quantity is invariant under rescalings of the metric, we may set $c_3 =1$ and rearrange terms to find that it becomes
\begin{equation}
I_1(c_1,c_2)= 3+\frac{(1-c_1^2 -c_2^2)^2}{(c_1 c_2)^{4/3}}-4(c_1 c_2)^{2/3}.
\end{equation}
In the first quadrant $c_1, c_2 >0$ this function has a local minimum $I_1(1,1) =0$, which corresponds to the round metric, and there are no other critical points on this domain. Limits to the $c_1$ and $c_2$-axes are infinitely positive, except at the two points $(c_1,c_2)=(1,0)$ or $(0,1)$ where the limit is 3. Moreover, limits to infinity are also infinitely positive, except in one direction when $c_1 =c_2$ where the limit is 3. It follows again that \eqref{monotonicity} is nonnegative, and vanishes only for a round $\Sigma$.
\end{proof}

This proposition shows that the first line of \eqref{dmdt} is nonnegative along a $\mathrm{SU}(n+1)$-invariant inverse mean curvature flow. When combined with weak nonnegativity of the scalar curvature for the Jang deformation, it will lead to an averaged monotonicity for Hawking mass in the Jang setting. The next result gives the expected upper bound for the limiting Hawking mass in terms of the ADM energy $E$. In dimension three this has been established \cite[Proposition 4.52]{Leetext}, \cite[Lemma 7.4]{Huisken:2001} using the Gauss-Bonnet Theorem. Here we avoid the need for Gauss-Bonnet with help from the proof of Proposition \ref{prop:monotonicity}.

\begin{prop}\label{convergence1}
Let $(M^{2(n+1)},g)$ be an asymptotically flat $\mathrm{SU}(n+1)$-invariant Riemannian manifold with outermost minimal surface boundary, for $n\geq 1$. If $\Sigma_s$ denotes the surface of distance $s$ from the boundary, then
\begin{equation}
\lim_{s\rightarrow\infty}\mathbf{m}_{H}(\Sigma_s)\leq E.
\end{equation}
\end{prop}

\begin{proof}
We will closely follow the arguments of \cite[Proposition 4.52]{Leetext}. Taking two traces of the Gauss equations along $\Sigma_s$ yields
\begin{equation}
H^2 = 2 \text{Ric}(\nu,\nu) - R + R_\Sigma + |II|^2,
\end{equation}
where $\nu$ is the unit normal pointing towards infinity. Let $d=2(n+1)$ and note that by Cauchy-Schwarz $(d-1)|II|^2 \geq H^2$, therefore
\begin{equation}
-(d-2) H^2 \leq -(d-1)R_\Sigma - 2(d-1)\text{Ric}(\nu,\nu) + (d-1)R.
\end{equation}
We then have
\begin{align}
\begin{split}
\mathbf{m}_H(\Sigma_s) \leq & \frac{(d-1)}{2(d-2)} \left(\frac{|\Sigma_s|}{\omega_{d-1}}\right)^{\frac{d-2}{d-1}}
\left[\frac{d-2}{d-1} - \frac{1}{(d-1)^2 \omega_{d-1}^{\frac{2}{d-1}} |\Sigma_s|^{\frac{d-3}{d-1}} }\int_{\Sigma_s} R_\Sigma dV\right] \\
&-\frac{1}{(d-2)(d-1)\omega_{d-1}} \left(\frac{|\Sigma_s|}{\omega_{d-1}}\right)^{\frac{1}{d-1}} \int_{\Sigma_s} G(\nu,\nu) dV,
\end{split}
\end{align}
where $G=\mathrm{Ric}-\tfrac{1}{2}Rg$ is the Einstein tensor. According to the discussion in Section \ref{sec2} it holds that
\begin{equation}
\left(\frac{|\Sigma_s|}{\omega_{d-1}}\right)^{\frac{1}{d-1}} =s+O(s^{1-\tau}),
\end{equation}
and hence \cite[Theorem 3.14]{Leetext} implies the last term converges to the ADM energy to produce
\begin{equation}
\lim_{s \to \infty} \mathbf{m}_H(\Sigma_s) \leq \lim_{s \to \infty} \frac{1}{2(d-1)} \left(\frac{|\Sigma_s|}{\omega_{d-1}}\right)^{\frac{d-2}{d-1}}I_n(s) + E,
\end{equation}
where $I_n(s)$ is the function from the proof of Proposition \ref{prop:monotonicity}. Since $B=O(s^{-\tau})$ for $n>1$ and $c_i=1+O(s^{-\tau})$ for $n=1$, a calculation shows that $I_n(s)=O(s^{-2\tau})$. Hence
\begin{equation}
\lim_{s \to \infty} \mathbf{m}_H(\Sigma_s) \leq E+ \lim_{s \to \infty} O(s^{d-2-2\tau})=E,
\end{equation}
as $\tau>\frac{d-2}{2}$.
\end{proof}

We are now in a position to complete the proof of Theorem \ref{foiqnoinoiqnhh}. Let  $(M^{2(n+1)},\bar{g})$ be the Jang manifold associated with the given initial data set, which is constructed from the solution given in Theorem \ref{jangexist} with $v(0)=\pm 1$ if $\partial M^{2(n+1)}$ is a past (future) apparent horizon having $H(0)\neq 0$, or with $v(0)=0$ if $H(0)=0$; in this last case the boundary is both a past and future apparent horizon. Note that this manifold is asymptotically flat, $\mathrm{SU}(n+1)$-invariant, and possesses an outermost minimal surface boundary. This last assertion follows from the fact that the $\bar{g}$-mean curvature of surfaces having constant distance to the boundary is $\bar{H}=\sqrt{1-v^2}H$, so that $\bar{H}(0)=0$ and $\bar{H}(\bar{s})>0$ for $\bar{s}>0$. Moreover, these surfaces provide a smooth inverse mean curvature $\{\Sigma_{\bar{t}}\}_{\bar{t}=0}^{\infty}$ in the Jang manifold. Next, recall that the scalar curvature of the Jang metric \cites{Bray:2009oni,Bray:2009au} takes the form
\begin{equation}\label{barscalar}
\bar{R}=16\pi(\mu-J(w))+|h-k|_{\bar{g}}^2 +2|X|_{\bar{g}}^2 -2\phi^{-1}\mathrm{div}_{\bar{g}}(\phi X),
\end{equation}
where
\begin{equation}
w_i=\frac{\phi f_i}{\sqrt{1+\phi^2 |\nabla f|_g^2}},\quad\quad h_{ij}=\frac{\phi \nabla_{ij} f + \phi_i f_j +\phi_j f_i}{\sqrt{1 + \phi^2 |\nabla f|_g^2}},\quad\quad
X_i=\frac{\phi f^j}{\sqrt{1+\phi^2 |\nabla f|_g^2}}(h_{ij}-k_{ij}).
\end{equation}
Here $h$ is the second fundamental form of the Jang graph in a dual Lorentzian setting, and $|w|_g\leq 1$. Although the computation appearing in \cites{Bray:2009oni,Bray:2009au} was set in dimension three, it extends without modification to all higher dimensions.
We may now integrate \eqref{dmdt} from $\bar{t}=0$ to $\infty$, apply the coarea formula as well as Propositions \ref{prop:monotonicity} and \ref{convergence1}, to find
\begin{align}\label{898989}
\begin{split}
\bar{E}- \frac{1}{2} \left(\frac{\mathcal{A}}{\omega_{2n+1}}\right)^{\frac{2n}{2n+1}}\geq &
\frac{1}{2(2n+1)^2 \omega_{2n+1}}\int_{M^{2(n+1)}} \left( \frac{|\Sigma_{\bar{t}}|}{\omega_{2n+1}}\right)^{\frac{1}{2n+1}}\bar{H}\bar{R}dV_{\bar{g}}\\
\geq& \frac{1}{(2n+1) \omega_{2n+1}}\left(\int_{\partial M^{2(n+1)}}\phi X(\bar{\nu})dV-\int_{\Sigma_{\infty}}\phi X(\bar{\nu})dV\right)
\end{split}
\end{align}
where $\bar{E}$ is the ADM energy with respect to the Jang metric and $\mathcal{A}=|\partial M^{2n+1}|$.
In the second inequality the dominant energy condition used, as well as the expression \eqref{fouqoboiqh} for $\phi$ together with the divergence theorem for the last term in \eqref{barscalar}. Here $\bar{\nu}=\partial_{\bar{s}}=\sqrt{1-v^2}\partial_s$, and $\Sigma_{\infty}$ indicates a limit to the asymptotic end along $s$-level sets.

Consider now the boundary terms of \eqref{898989}. According to \cite[pg. 582]{Bray:2009oni}, the Jang equation implies that
\begin{equation}
(h -k)(\nu,\nu)=(1+\phi^2 |\nabla f|_g^2)\left(-\frac{\phi \nu(f)}{\sqrt{1+\phi^2 |\nabla f|_g^2}}H+\mathrm{Tr}_{\Sigma}k\right)=(1-v^2)^{-1}(-v H+\mathrm{Tr}_{\Sigma}k)
\end{equation}
for any level set $\Sigma$ of $f$ with unit normal $\nu$. Therefore, on each $\Sigma_s$ it holds that
\begin{equation}\label{0987}
\phi X(\bar{\nu})= \phi \sqrt{1-v^2} \left(\frac{\phi f'}{\sqrt{1+\phi^2 f'^2}}\right)(h_{ss}-k_{ss})
=v\rho'\left((\pm 1-v)H \mp \theta_{\mp}\right).
\end{equation}
Since $|v(s)|=O(s^{-2\tau_0})$ for $2\tau_0>d-2=2n$ and $H,\theta_{\pm}=O(s^{-1})$, the boundary integral at infinity converges to zero. Moreover, 
if the boundary is a past (future) apparent horizon with $H(0)\neq 0$ then 
$v(0)=\pm 1$, and if the boundary is an apparent horizon with $H(0)=0$ then $v(0)=0$. It follows that the inner boundary integral vanishes in all cases. The desired Penrose inequality \eqref{foiaoinoqinhoipqjh} involving ADM mass now follows from the agreement between the ADM energies $\bar{E}=E$ in light of the decay of $v$, and the vanishing ADM linear momentum as shown in Appendix \ref{appvanishk}. 

It remains to establish the rigidity statement. If equality holds in \eqref{foiaoinoqinhoipqjh}, then equality must hold in Proposition \ref{prop:monotonicity} for each surface $\Sigma_s$. This shows that each such surface is round, implying that the original initial data set $(M^{2(n+1)},g,k)$ is spherically symmetric. In this case it is known that the Jang graph yields the desired embedding into the Schwarzschild spacetime \cite[Theorem 3.4]{BKS}.

\section{Proof of Theorem \ref{foiqnoinoiqnhh1}: Asymptotically Hyperbolic Penrose Inequality}
\label{sec5} \setcounter{equation}{0}
\setcounter{section}{5}

In this section we will show how the arguments presented for the asymptotically flat case can be modified to establish the spacetime Penrose inequality in the asymptotically hyperbolic setting. A 3-dimensional `hyperbolic Hawking mass' was defined by Bray-Chu\'{s}ciel in \cite[Section 4.1]{Bray:2003ns}, and its generalization to $d$-dimensions may be expressed as
\begin{equation}\label{Hawkingmass:hyp}
\mathbf{m}^{\mathrm{hyp}}_H(\Sigma) := \frac{1}{2} \left(\frac{|\Sigma|}{\omega_{d-1}}\right)^{\frac{d-2}{d-1}}\left[ 1 + \left(\frac{|\Sigma|}{\omega_{d-1}}\right)^{\frac{2}{d-1}} - \frac{1}{(d-1)^2 \omega_{d-1}^{\frac{2}{d-1}} |\Sigma|^{\frac{d-3}{d-1}} }\int_\Sigma H^2 dV \right].
\end{equation}  
Under a smooth inverse mean curvature flow its first variation is given by
\begin{equation} \label{dmdthyp}
\begin{aligned}
\frac{d}{dt} \mathbf{m}^{\mathrm{hyp}}_H &(\Sigma_t)  = \frac{1}{2} \left( \frac{|\Sigma_t|}{\omega_{d-1}}\right)^{\frac{d-2}{d-1}} \left( \frac{d-2}{d-1} - \frac{1}{(d-1)^2 \omega_{d-1}^{\frac{2}{d-1}} |\Sigma_t|^{\frac{d-3}{d-1}}} \int_{\Sigma_t} R_\Sigma dV\right) \\
&+  \frac{1}{2} \left( \frac{|\Sigma_t|}{\omega_{d-1}}\right)^{\frac{d-2}{d-1}}\!\!\!  \frac{1}{(d-1)^2 \omega_{d-1}^{\frac{2}{d-1}} |\Sigma_t|^{\frac{d-3}{d-1}}} \int_{\Sigma_t} \!\left( \frac{|\nabla_\Sigma H|^2}{H^2}\! + \!|II|^2 \!-\! \frac{H^2}{d-1} + R + d(d-1)\!\right)\! dV.
\end{aligned}
\end{equation} 
Note that this differs from the asymptotically flat formula only by the addition of the last term involving $d(d-1)$, which is relevant for the hyperbolic dominant energy condition. In the $\mathrm{SU}(n+1)$-invariant case, Proposition \ref{prop:monotonicity} implies that the first line is again nonnegative, moreover the following analogue of Proposition \ref{convergence1} yields the appropriate asymptotics for this quasi-local mass. 

\begin{prop}\label{convergence2}
Let $(M^{2(n+1)},g)$ be an asymptotically hyperbolic $\mathrm{SU}(n+1)$-invariant Riemannian manifold with outermost minimal surface boundary, for $n\geq 1$. If $\Sigma_s$ denotes the surface of distance $s$ from the boundary, then
\begin{equation}
\lim_{s\rightarrow\infty}\mathbf{m}^{\mathrm{hyp}}_{H}(\Sigma_s)= E_{\mathrm{hyp}}.
\end{equation}
\end{prop}

\begin{proof}
We will assume that $n>1$, as the case $n=1$ may be treated similarly. The hyperbolic defect tensor from the definition of hyperbolic mass \eqref{CHenergy}, when expressed in the coordinates provided by Proposition \ref{aofnoaigopinjopinhipoq}, is given by
\begin{equation}
\mathfrak{g}=g-b=\left(\rho^2 e^{-4nB} -\sinh^2 s\right)(d\psi +A)^2 +\left(\rho^2 e^{2B}-\sinh^2 s\right)g_{FS}.
\end{equation}
We proceed to compute the relevant terms of \eqref{CHenergy}. Note that $\mathfrak{g}(\nu_b,\cdot)=0$, and therefore 
\begin{equation}
(\mathrm{div}_{b}\mathfrak{g})(\nu_b)=\nu_b^j \pmb{\nabla}^i \mathfrak{g}_{ij}=-(\pmb{\nabla}^i \nu_b^j)\mathfrak{g}_{ij}
=-(\coth s)\mathrm{Tr}_b \mathfrak{g},
\end{equation}
where $\pmb{\nabla}$ denotes covariant differentiation with respect to $b$. Moreover
\begin{equation}
\mathrm{Tr}_b \mathfrak{g}=-(2n+1) +\left(\frac{\rho}{\sinh s}\right)^2\left(e^{-4n B}+2ne^{2B}\right),
\quad\quad\quad \mathfrak{g}(\nabla_b U,\nu_b)=0,
\end{equation}
in which $U=\cosh s$ is the lapse function. A further calculation and rearrangement shows that the energy becomes
\begin{align}\label{foaiaoihogh}
\begin{split}
E_{\mathrm{hyp}}=&\frac{1}{2(d-1) \omega_{d-1}}\lim_{s\rightarrow\infty} \int_{\Sigma_s} \left[U(\text{div}_b \mathfrak{g}) - U (d \mathrm{Tr}_b \mathfrak{g}) + (\mathrm{Tr}_b \mathfrak{g}) d U - \mathfrak{g}(\nabla_b U)\right] (\nu_b) d V\\
=&\lim_{s\rightarrow\infty}\frac{(\sinh s)^{2n}}{2}\left[1-\gamma\left(\frac{\rho^2}{\sinh^2 s}-2\frac{\rho^2 \cosh^2 s}{\sinh^2 s}+2\frac{\rho \rho' \cosh s}{\sinh s}\right)\right]\\
&+\lim_{s\rightarrow\infty}\frac{2n\rho^2\cosh s (\sinh s)^{2n-1}}{(2n+1)}B' (e^{-4nB}-e^{2B}),
\end{split}
\end{align}
with
\begin{equation}
\gamma=\frac{e^{-4nB}+2ne^{2B}}{2n+1}=1 + O(e^{-2qs}).
\end{equation}
Observe that the last term of \eqref{foaiaoihogh} decays on the order of $O(e^{(2n+2-2q)s})$, and hence
\begin{equation}\label{fohqoinfoinqpoihnmqh}
E_{\mathrm{hyp}}\!=\!\lim_{s\rightarrow\infty}\!\left[\frac{(\sinh s)^{2n}}{2}\!\left(\!1+\rho^2 -\rho'^2 +\left(\!\rho'\!-\frac{\rho\cosh s}{\sinh s}\right)^2 \right)\!+O\!\left(e^{(2n+2-2q)s}\right)\right]
\!=\!\lim_{s\rightarrow\infty} m_{H}^{\mathrm{hyp}}(\Sigma_s),
\end{equation}
since
\begin{equation}\label{giobqoinoiqhh}
\mathbf{m}_{H}^{\mathrm{hyp}}(\Sigma_s)=\frac{\rho^{2n}}{2}(1+\rho^2 -\rho'^2)
\end{equation}
and $q>n+1$.
\end{proof}

We are now in a position to complete the proof of Theorem \ref{foiqnoinoiqnhh1}. As in the asymptotically flat case, let $(M^{2(n+1)},\bar{g})$ be the Jang manifold associated with the given initial data set, which is constructed from the solution given in Theorem \ref{jangexist} with $v(0)=\pm 1$ if $\partial M^{2(n+1)}$ is a past (future) apparent horizon having $H(0)\neq 0$, or with $v(0)=0$ if $H(0)=0$. The asymptotics of $v$ imply that $\bar{g}=g+O_2(e^{-2q_0 s})$ where $q_0 > n+1$, and hence the Jang manifold is asymptotically hyperbolic. It is also $\mathrm{SU}(n+1)$-invariant, has an outermost minimal surface boundary, and the surfaces $\{\Sigma_{\bar{t}}\}_{\bar{t}=0}^{\infty}$ of constant $\bar{g}$-distance from the boundary give a smooth inverse mean curvature flow. Utilizing \eqref{dmdthyp}, the hyperbolic dominant energy condition, and the weak nonnegativity of the Jang scalar curvature \eqref{barscalar}, together with the arguments leading to \eqref{898989} produces
\begin{equation}\label{89898912}
\bar{E}_{\mathrm{hyp}}-\frac{1}{2}\!\left(\frac{\mathcal{A}}{\omega_{2n+1}}\right)^{\frac{2n}{2n+1}}\!\!\!\! - \frac{1}{2} \!\left(\frac{\mathcal{A}}{\omega_{2n+1}}\right)^{\frac{2n+2}{2n+1}}\!
\!\geq \!\frac{1}{(2n+1) \omega_{2n+1}^{\frac{2}{2n+1}}}\left(\int_{\partial M^{2(n+1)}}\!\!\!\!\!\!\phi X(\bar{\nu})dV-\!\int_{\Sigma_{\infty}}\!\!\phi X(\bar{\nu})dV\!\right),
\end{equation}
where $\bar{E}_{\mathrm{hyp}}$ is the total energy with respect to the Jang metric and Propition \ref{convergence2} has been employed. The boundary integral vanishes for the same reasons as presented in the proof of Theorem \ref{foiqnoinoiqnhh}. Furthermore the decay recorded in Section \ref{sec2} and \eqref{vfall} imply 
\begin{equation}\label{0987hyp}
\phi X(\bar{\nu})= v\rho'(-vH+\mathrm{Tr}_{\Sigma}k)
=O\left(e^{(1-2q_0)s}\right).
\end{equation}
It follows that the last term in \eqref{89898912} is also zero, since the integrals of $\Sigma_s$ fall-off on the order of $O(e^{(2n+2-2q_0)s})$. Therefore, if the original energy $E_{\mathrm{hyp}}$ agrees with that of the Jang manifold $\bar{E}_{\mathrm{hyp}}$, then \eqref{fohqoinfoinqpoihnmqh} confirms the desired Penrose inequality \eqref{hypPenrose}. To see that this is indeed valid, first observe that 
\begin{equation}
\left(\frac{d\rho}{d\bar{s}}\right)^2=(1-v^2)\left(\frac{d\rho}{ds}\right)^2 =\left(\frac{d\rho}{ds}\right)^2 +O\left(e^{(2-2q_0) s}\right),
\end{equation}
and then apply \eqref{giobqoinoiqhh} as well as Proposition \ref{convergence2} to find
\begin{equation}
\bar{E}_{\mathrm{hyp}}=\lim_{\bar{s}\rightarrow\infty} \bar{\mathbf{m}}_{H}^{\mathrm{hyp}}(\Sigma_{\bar{s}})
=\lim_{s\rightarrow\infty}\left(\mathbf{m}_{H}(\Sigma_s)+O\left(e^{(2n+2-2q_0) s}\right)\right)=E_{\mathrm{hyp}},
\end{equation}
where $\bar{\mathbf{m}}_{H}^{\mathrm{hyp}}$ denotes Hawking mass with respect to the Jang metric.

It remains to establish the rigidity statement. If equality holds in \eqref{hypPenrose}, then equality must hold in Proposition \ref{prop:monotonicity} for each surface $\Sigma_{\bar{s}}$. This shows that each such surface is round, implying that the Jang manifold is spherically symmetric. Furthermore, from \eqref{dmdthyp} it follows that $\bar{R}=-(2n+1)(2n+2)$ is constant, and we conclude that $(M^{2(n+1)},\bar{g})$ is isometric to a constant time slice of the Schwarzschild-AdS spacetime. Therefore, \eqref{fouqoboiqh} implies the following expression for the change of raidal coordinates
\begin{equation}
\bar{g}=d\bar{s}^2 +\rho^2(\bar{s}) g_{S^{2n+1}}=\frac{d\rho^2}{\phi^2(\rho)}+\rho^2 g_{S^{2n+1}}
\end{equation}
where $g_{S^{2n+1}}$ is the unit sphere metric, with
\begin{equation}
\phi^2(\rho)=1-\frac{2E_{\mathrm{hyp}}}{\rho^{2n}}+\rho^2.
\end{equation}
Note that this formula $\phi$ may also be found from \eqref{giobqoinoiqhh}, together with the fact that the Hawking mass is constant $\bar{\mathbf{m}}_{H}^{\mathrm{hyp}}(\Sigma_{\bar{s}})=E_{\mathrm{hyp}}$.
The original metric $g=\bar{g}-\phi^2 df^2$ is then induced from the graph (over a constant time slice) given by $f$ inside Schwarzschild-AdS. Lastly, the second fundamental form of this isometric embedding agrees with $k$ due to the fact that the second term (in fact each term except for $\mu$) of \eqref{barscalar} vanishes.

\section{Proof of Theorem \ref{thm3}: All Cohomogeneity One Cases}
\label{sec6} \setcounter{equation}{0}
\setcounter{section}{6}

In this section we will establish the Riemannian Penrose inequality for all cohomogeneity one manifolds. According to the discussion at the beginning of Section \ref{sec2}, in addition to the $\mathrm{SU}(n+1)$-invariant setting considered above there are two other cases to consider, namely the $\mathrm{Sp}(n+1)$ and $\mathrm{Spin}(9)$-invariant cases. Let $(M^d,g)$, $d\geq 3$ be an asymptotically flat cohomogeneity one Riemannian manifold with outermost minimal surface boundary and nonnegative scalar curvature. As previously mentioned, $M^d \cong [0,\infty) \times S^{d-1}$ and $g=ds^2 +g_s$ where $s$ is the distance function to the boundary, and $g_s$ is a 1-parameter family of $\mathcal{G}$-invariant metrics on $S^{d-1}$. Interestingly, when $\mathcal{G}$ is either 
$\mathrm{Sp}(n+1)$ or $\mathrm{Spin}(9)$, the analogue of Proposition \ref{prop:monotonicity} does not hold in general, and therefore the inverse mean curvature flow approach breaks down. To deal with this issue, we utilize the conformal flow method of Bray \cite{Bray} until the outermost minimal surface reaches an appropriate location in the asymptotic end. From there, a modified version of this proposition may be implemented to complete the argument with inverse mean curvature flow. This bypasses the more involved problem of showing that the conformal flow converges to Schwarzschild in the asymptotic end.

\subsection{$\text{Sp}(n+1)$-invariant metrics}

Let $d=4n+4$ for $n\geq 0$. The group $\mathcal{G}= \mathrm{Sp}(n+1)$ acts transitively on the sphere $S^{4n+3}$ with isotropy subgroup $\mathcal{H} = \mathrm{Sp}(n)$. Any $\mathcal{G}$-invariant metric $g_{\mathbf{c}}$ on this sphere depends (up to an overall scaling) on three positive parameters $\mathbf{c}=(c_1, c_2, c_3)$, and arises as a Riemannian submersion for the Hopf fibration
$S^3 \hookrightarrow S^{4n+3} \to \mathbb{HP}^n$ in which the parameters scale the fiber directions and the base is equipped with the canonical Einstein metric on quaternionic projective space. In this notation, the round metric of unit curvature is then described by $\mathbf{c}=(1,1,1)$. Moreover, the $s$-level set metrics from the discussion above are then given by $g_s=\rho^2 (s) g_{\mathbf{c}(s)}$ for some positive function $\rho(s)$. According to \cite[Section 8]{Betencourt:2021} the scalar curvature of $g_{\mathbf{c}}$ takes the form
\begin{align}
\begin{split}
R_{\mathbf{c}} =& \frac{2}{c_1 c_2 c_3}\left(c_1^2 + c_2^2 + c_3^2 - (c_2-c_3)^2 - (c_3 - c_1)^2 - (c_1 - c_2)^2\right)\\
&- 4n (c_1 + c_2 + c_3) + 16 n^2 + 32n.
\end{split}
\end{align}
Notice that setting $c_i =c^2$ yields
\begin{equation}
R_{\mathbf{c}}=\frac{6}{c^2}-12n c^2+16n^2+32n,
\end{equation}
showing that a collapse of the Hopf fibers results in blow-up of curvature, which is in contrast to the $\mathrm{SU}(n+1)$-invariant setting. Furthermore due to this behavior, Proposition \ref{prop:monotonicity} does not hold here without additional qualification. The next result provides a regime for which the result remains valid.

\begin{prop}\label{prop:monotonicity1}
Let $(M^{4(n+1)},g)$ be an asymptotically flat $\mathrm{Sp}(n+1)$-invariant Riemannian manifold with outermost minimal surface boundary, for $n\geq 0$. There exists a distance $s_0>0$ such that if $\Sigma_s$ denotes the surface of distance $s$ to the boundary and $s\geq s_0$, then
\begin{equation}
\int_{\Sigma_s}R_{\Sigma}dV\leq C_n |\Sigma_s|^{\frac{4n+1}{4n+3}}
\end{equation}
where $C_n = (4n+2)(4n+3)\omega_{4n+3}^{\frac{2}{4n+3}}$, with equality achieved in this regime if and only if $\Sigma_s$ is a round sphere.
\end{prop}

\begin{proof}
Since the volume form for $g_{\mathbf{c}}$ is given by $dV=(c_1 c_2 c_3)^{1/2}dV_{S^{4n+3}}$, we have
\begin{equation}
I(\mathbf{c}):= (4n+2)(4n+3) -\frac{1}{\omega_{4n+3}^{\frac{2}{4n+3}}|\Sigma_s|^{\frac{4n+1}{4n+3}}}\int_{\Sigma_s}R_{\Sigma}dV=(4n+2)(4n+3)-(c_1 c_2 c_3)^{1/(4n+3)}R_{\mathbf{c}}.
\end{equation}
A direct calculation shows that $\mathbf{c}=(1,1,1)$ is a local isolated minimum for the function $I$.
Thus, since $\Sigma_s$ uniformly approaches the unit round sphere as $s\rightarrow\infty$, there exists $s_0$ such that $s\geq s_0$ implies $I(\mathbf{c(s)})\geq I(1,1,1)$ with equality if and only if $g_s$ is a round metric. Moreover, as in the proof of Proposition \ref{aofnoaigopinjopinhipoq}, asymptotic flatness produces
\begin{equation}\label{uiuik1}
\rho(s)=s+O_2(s^{1-\tau}),\quad\quad\quad\quad c_i(s)=1+O_2(s^{-\tau}),\quad i=1,2,3.
\end{equation}
Hence, since $I(1,1,1)=|\nabla I(1,1,1)|=0$ we have $I(\mathbf{c(s)})=O(s^{-2\tau})$ as $s\rightarrow\infty$.
\end{proof}

\subsection{$\text{Spin}(9)$-invariant metrics} 

The group $\mathcal{G}=\mathrm{Spin}(9)$ acts transitively on $S^{15}$ with isotropy subgroup $\mathcal{H}=\mathrm{Spin}(7)$.  Any $\mathcal{G}$-invariant metric $g_{c}$ on this sphere depends (up to an overall scaling) on one positive parameters $c$, and arises as a Riemannian submersion for the Hopf fibration
$S^7 \hookrightarrow S^{15} \to S^8$ in which the parameter scales the fiber and the base is equipped with the unit round metric. In this notation, the round metric of unit curvature on $S^{15}$ is then described by $c=1$. Moreover, the $s$-level set metrics from the discussion at the beginning of this section are then given by $g_s=\rho^2 (s) g_{c(s)}$ for some positive function $\rho(s)$. According to \cite[Section 8]{Betencourt:2021} the scalar curvature of $g_{c}$ takes the form
\begin{equation}
R_{c} = \frac{42}{c} - 56 c + 224.
\end{equation}
Notice that as in the $\mathrm{Sp}(n+1)$-invariant case, collapse of the Hopf fibers results in blow-up of curvature, again in contrast to the $\mathrm{SU}(n+1)$-invariant setting. Thus, we must again replace Proposition \ref{prop:monotonicity} with an asymptotic version.

\begin{prop}\label{prop:monotonicity2}
Let $(M^{16},g)$ be an asymptotically flat $\mathrm{Spin}(9)$-invariant Riemannian manifold with outermost minimal surface boundary. There exists a distance $s_0>0$ such that if $\Sigma_s$ denotes the surface of distance $s$ to the boundary and $s\geq s_0$, then
\begin{equation}
\int_{\Sigma_s}R_{\Sigma}dV\leq \mathcal{C} |\Sigma_s|^{\frac{13}{15}}
\end{equation}
where $\mathcal{C} = 210\omega_{15}^{\frac{2}{15}}$, with equality achieved in this regime if and only if $\Sigma_s$ is a round sphere.
\end{prop}

\begin{proof}
Since the volume form for $g_{c}$ is given by $dV=c^{7/2}dV_{S^{15}}$, we have
\begin{equation}
I(c):= 210 -\frac{1}{\omega_{15}^{\frac{2}{15}}|\Sigma_s|^{\frac{13}{15}}}\int_{\Sigma_s}R_{\Sigma}dV=210-c^{7/15}R_{c}.
\end{equation}
A direct calculation shows that $c=1$ is a local isolated minimum for the function $I$.
Thus, since $\Sigma_s$ uniformly approaches the unit round sphere as $s\rightarrow\infty$, there exists $s_0$ such that $s\geq s_0$ implies $I(c(s))\geq I(1,1,1)$ with equality if and only if $g_s$ is a round metric. Moreover, as in the proof of Proposition \ref{aofnoaigopinjopinhipoq}, asymptotic flatness produces
\begin{equation}\label{uiuik2}
\rho(s)=s+O_2(s^{1-\tau}),\quad\quad\quad\quad c(s)=1+O_2(s^{-\tau}),\quad i=1,2,3.
\end{equation}
Hence, since $I(1)=|I'(1)|=0$ we have $I(c(s))=O(s^{-2\tau})$ as $s\rightarrow\infty$.
\end{proof}

\subsection{Combining the conformal flow with inverse mean curvature flow}

Let $(M^d,g)$ be as described at the start of this section, with $\mathcal{G}=\mathrm{Sp}(n+1)$ or $\mathrm{Spin}(9)$. Consider the conformal flow of metrics defined by $g_t=u_t^{\frac{4}{n-2}}g$ where $\tfrac{d}{dt}u_t=v_t$ satisfies
\begin{equation}\label{deltagv}
\Delta_g v_t =0 \quad\text{ on }\quad M^{d}_t, \quad\quad v_t =0 \quad\text{ on }\quad \partial M^d_t,
\quad\quad v_t(x)\rightarrow -e^{-t} \quad\text{ as }\quad |x|\rightarrow\infty,
\end{equation}
and $v_t=0$ on $M^d \setminus M^d_t$. Here $M^d_t$ denotes the region outside of the outermost minimal surface (denoted $\partial M^d_t$) in $(M^d, g_t)$. The conformal flow was initially studied \cite{Bray} in dimension 3, and was extended to higher dimensions in \cite{Bray:2007opu}. In particular, the flow exists as long as the outermost minimal surfaces involved remain smooth. In the current cohomogeneity one setting, the relevant minimal surfaces $\partial M^d_t$ must be level sets of the distance function to $\partial M^d$, and are therefore smooth. Moreover, the functions $v_t$ as well as $u_t$ depend only on $s$, showing that the conformal metrics $g_t$ are also $\mathcal{G}$-invariant. According to \cite[Lemma 2.3]{Bray:2007opu}, existence of the flow guarantees that the areas $|\partial M_t^d|$ remain constant in $t$. Furthermore, the mass decrease law, which was partially responsible for the dimensional restriction in \cite{Bray:2007opu} due to its reliance on the positive mass theorem, is valid here since each $M^d_t$ is a spin manifold. More precisely, the positive mass theorem is applied to a doubled manifold with one of the two ends being compactified, and in this context we may apply the `corners' version of this result obtained with harmonic spinors \cite[Theorem 3.1]{ShiTam}. It follows that the mass $m(t)$ of $(M^d_t,g_t)$ is nonincreasing in $t$.

We now claim that the flow surfaces $\partial M^d_t$ reach the asymptotic end of $M^d$ in finite time, and in fact eventually leave every compact set. This was originally established in dimension 3 \cite[Theorem 13]{Bray}, while in \cite{Bray:2007opu} this issue is avoided altogether. Although the original proof relied on the Gauss-Bonnet theorem which is not available here, a weaker version of this result \cite[Theorem 12]{Bray} extends to higher dimensions with only minor changes. An immediate corollary of this theorem shows that $\partial M^d_t$ cannot be entirely enclosed by a single coordinate sphere in the asymptotic end, for all $t$. Because the flow surfaces in the current setting are $\mathcal{G}$-invariant, if they did not eventually leave every compact set then they would be entirely enclosed by a coordinate sphere. Thus, we may apply the higher dimensional analogue of Bray's observation to obtain the desired conclusion. Note that the proof of \cite[Theorem 12]{Bray} makes use of harmonic asymptotics, however this is not necessary as the relevant harmonic functions may be expanded in spherical harmonics to produce the same outcome.

To complete the proof, run the conformal flow until time $t_0$ when $\partial M^d_{t_0}$ reaches sufficiently far out in the asymptotic end to enclose $\Sigma_{s_0}$, the designated surface appearing in Propositions \ref{prop:monotonicity1} and \ref{prop:monotonicity2}. Properties of the flow discussed above guarantee that
\begin{equation}\label{ofinqoingoipnqh}
m\geq m(t_0),\quad\quad\quad\quad |\partial M^d_{t_0}|=\mathcal{A},
\end{equation}
where $m$ is the ADM mass of $(M^d,g)$ and $\mathcal{A}$ is the area of its boundary. On the other hand, the aforementioned propositions give rise to monotonicity for the Hawking mass in $(M^d_{t_0},g_{t_0})$, as in Section \ref{sec4}. Moreover, since the functions $I$ from the proof of these propositions vanishes to second order at the round metric, the asymptotic limit of the Hawking mass is no greater than $m(t_0)$; this is established in the same manner as Proposition \ref{convergence1}. Therefore
\begin{equation}\label{fvlkanognopiqnhq}
m(t_0)\geq \frac{1}{2}\left(\frac{|\partial M^d_{t_0}|}{\omega_{d-1}}\right)^{\frac{d-2}{d-1}},
\end{equation}
and combining \eqref{ofinqoingoipnqh} with \eqref{fvlkanognopiqnhq} produces the desired Penrose inequality since $\mathcal{A}_h =\mathcal{A}$. Lastly, consider the case of equality for \eqref{ofinaoinoinoiqhhq}. 
This forces equality between the masses of \eqref{ofinqoingoipnqh}, and thus the rigidity statement of the positive mass theorem used for the (conformal flow) mass decrease law, implies that $(M^d_t, g_t)$ is  spherically symmetric for $t\leq t_0$. In particular, $(M^d,g)$ is spherically symmetric and is therefore isometric to the canonical slice of a Schwarzschild spacetime \cite[Theorem 3.4]{BKS}.

\appendix

\section{Vanishing of Linear Momenta}\label{appvanishk}

In this appendix we show that under the hypotheses of Theorem \ref{foiqnoinoiqnhh}, the ADM linear momenta $P_i$ vanish so that the ADM mass agrees with the ADM energy $m=E$. Let $n>1$ and use the expressions for the metric and extrinsic curvature given in Proposition \ref{aofnoaigopinjopinhipoq}, to compute the following two components of the momentum density 
\begin{equation}\label{alkfjlkaj}
      J(\mathbf{e}_1) = J(\partial_s)  = -k'_b - 2n k'_c + \frac{(2n+1) \dot\rho}{\rho} k_a - (k_b + 2n k_c) \frac{\rho'}{\rho} + 2n B' (k_b - k_c), 
\end{equation}
\begin{equation}\label{oginqoingoinqh}
      J(\mathbf{e}_2) = k'_s + k_s \left((2n+2) \frac{\dot \rho}{\rho} - 2n B' \right),
\end{equation}
where $\mathbf{e}_1$ and $\mathbf{e}_2$ are a part of the orthonormal frame from \eqref{foanofinh}. Recalling the formula for $\mathrm{Tr}_g k$ in \eqref{tracek12} motivates a rewriting of \eqref{alkfjlkaj} by
\begin{equation}
k'_a +(2n+2)\frac{\rho'}{\rho}k_a=J(\mathbf{e}_1)+\left(\mathrm{Tr}_g k\right)'+\frac{\rho'}{\rho}\mathrm{Tr}_g k -2n B' (k_b -k_c).
\end{equation}
Since the asymptotically flat fall-off conditions \eqref{1.3} imply that the right-hand side is $O(s^{-2\tau-2})$, it follows that
\begin{equation}
k_a=\rho^{-2\tau-2}\left[\int_{s_0}^{s}\rho^{2\tau+2}O(t^{2\tau+2})dt+C\right]=O(s^{-2\tau-2})
\end{equation}
for some constant $C$. Similarly, we also find that $k_s=O(s^{-2\tau-2})$ from \eqref{oginqoingoinqh}. Moreover, because the only possible nonzero components of the ADM linear momentum involve only $k(e_1,e_1)=k_a$ or $k(e_1,e_2)=k_s$, and $\tau>n$, we conclude that the ADM linear momentum vanishes. The case $n=1$ is treated analogously using Section \ref{section2.2}.
Note that the additional decay of $\mathrm{Tr}_g k$ from \eqref{1.3} is only used here, to obtain the vanishing linear momentum.

\section{An Example}
\label{appA}

In this last appendix we will exhibit cohomogeneity one asymptotically flat and asymptotically hyperbolic initial data, which deviate in a significant manner from spherical symmetry by allowing for nonzero angular momentum.  Consider the 2-parameter family of asymptotically flat $\mathrm{SU}(n+1)$-invariant initial data $([r_+,\infty)\times S^{2n+1},g,k)$ in which
\begin{equation}\label{MP(AdS)}
g = U(r)^2 dr^2 + P(r)^2 \left( d \psi + A \right)^2 + r^2 g_{FS} ,\quad\quad\quad
k = -r^{-1}U(r) W'(r)P(r)^3 dr(d \psi + A),
\end{equation} 
where $\psi$, $A$, and $g_{FS}$ are as in \eqref{Berger} and
\begin{equation}
U(r)^2 = \left(1 - \frac{2m}{r^{2n}} + \frac{2m a^2}{r^{2n+2}} \right)^{-1}, \quad\quad P(r) = r \left( 1 + \frac{2m a^2}{r^{2n+2}}\right)^{1/2}, \quad\quad W(r) = \frac{2m a}{r^{2n} P(r)^2},
\end{equation} 
with $m$ and $a$ denoting the mass and angular momentum parameters. The value $r_+$ is the largest positive root of $U(r)^{-2}$, and corresponds to an outermost minimal surface. This initial data set arises from the 2-parameter family of Myers-Perry stationary asymptotically flat black hole solutions of the vacuum Einstein equations given by
\begin{equation}
\mathbf{g} = -r^2 U(r)^{-2} P(r)^{-2}d t^2 + U(r)^2 dr^2 + P(r)^2 \left( d \psi + A - W(r) d t\right)^2 + r^2 g_{FS}.
\end{equation} 
Note that when $a=0$ the Schwawrzschild solution is recovered. The event horizon is located at $r = r_+$ and has  null generator
\begin{equation}
\xi = \partial_t + \Omega \partial_\psi, \qquad\quad \Omega = \frac{2m a}{r_+^{2(n+1)} + 2 m a^2},
\end{equation} 
where $\Omega$ is the angular velocity. Moreover, a calculation shows that for the black hole to be subextremal
the parameters must satisfy $r_+^2 > n^{-1}(n+1)a^2$, with equality corresponding to an extreme Myers-Perry solution having a degenerate horizon.  It is convenient to use $r_+$ and $a$ express relevant quantities. In particular, the mass parameter $m$ coincides with the ADM mass and takes the form 
\begin{equation}
m = \frac{r_+^{2(n+1)}}{2(r_+^2 - a^2)},
\end{equation}
while the area of a cross section of the event horizon becomes
\begin{equation}
\mathcal{A} = \omega_{2n+1} P(r_+) r_+^{2n} = \frac{\omega_{2n+1} r_+^{2(n+1)}}{\sqrt{r_+^2-a^2}},
\end{equation}
and the only nonzero ADM angular momentum occurs in the $\psi$-direction and is given by $\mathcal{J}_\psi = m a$ where by definition
\begin{equation}
\mathcal{J}_{\psi} = \frac{1}{2(n+1)\omega_{2n+1}}\int_{S_\infty} (k - (\text{Tr} k)g) (\partial_r ,\partial_\psi) dV.
\end{equation} 
We now have the relation
\begin{equation}
m = \left(\frac{r_+^2}{r_+^2 - a^2}\right)^{\frac{n+1}{2n+1}} \frac{1}{2} \left(\frac{\mathcal{A}}{\omega_{2n+1}}\right)^{\frac{2n}{2n+1}}.
\end{equation} 
Note that this satisfies the Penrose inequality since $a^2 r_{+}^{-2}<1$, and equality holds only for the case of Schwarzschild, when $a=0$.

An asymptotically hyperbolic generalization of the above rotating black hole initial data may be obtained by setting
\begin{equation}
U(r)^2 = \left(1 + r^2 - \frac{2m(1-a^2)}{r^{2n}} + \frac{2m a^2}{r^{2n+2}} \right)^{-1}.
\end{equation} 
The resulting data coincides with the canonical slice of a Myers-Perry-AdS black hole, a solution of the stationary vacuum Einstein equations with negative cosmological constant, having equal angular momenta. Again letting $r_+$ denote the largest positive root of $U(r)^{-2}$, there will be an event horizon at $r= r_+$ provided $m > 0$ and $0 \leq a < 1$. Notice that when $m =0$ the spacetime is the 
$(2n+3)$-dimensional AdS space, whereas when $a=0$ the solution reduces to Schwarzschild-AdS.  A computation reveals the total energy to be
\begin{equation}
E_{\mathrm{hyp}} = m \left( 1 + \frac{a^2}{2n+1} \right),
\end{equation} 
while as before the angular momentum is $\mathcal{J}_\psi = m a$. Expressing the mass parameter in terms of $r_+$ and $a$ yields
\begin{equation}
m = \frac{(1+r_+^2)r_+^{2(n+1)}}{2(r_+^2(1-a^2) - a^2)},
\end{equation} 
and therefore
\begin{equation}
m = \frac{1}{2\left(1-\alpha\right)^{\frac{n+1}{2n+1}} (1-\beta)} \left(\frac{\mathcal{A}}{\omega_{2n+1}}\right)^{\frac{2n}{2n+1}} + \frac{1}{2\left(1-\alpha\right)^{\frac{n}{2n+1}} (1-\beta)} \left(\frac{\mathcal{A}}{\omega_{2n+1}}\right)^{\frac{2n+2}{2n+1}}
\end{equation} 
where
\begin{equation}
\alpha=\frac{a^2}{r_+^2}, \quad\qquad \beta= \frac{a^2}{r_+^2(1-a^2)}.
\end{equation}
Next observe that a calculation provides the nondegeneracy condition 
\begin{equation}
r_+^2(n + (n+1)r_+^2) - a^2 (1+n)(1+r_+^2)^2 \geq 0.
\end{equation}
Hence, because restrictions on the parameters imply that $\alpha<1$ and $\beta<1$, and $E_{\mathrm{hyp}}\geq m$ with equality only when $a=0$, it follows that the hyperbolic Penrose inequality holds with saturation only for Schwarzschild-AdS.

\end{document}